\documentclass[10pt,a4paper,reqno]{amsart}            

\usepackage[english]{babel}
\usepackage{amssymb,amsmath,amsthm}
\usepackage{hyperref}
\usepackage{mathrsfs}

\renewcommand\mathcal{\mathscr}

\theoremstyle{plain}
\newtheorem{theorem}{Theorem}[section]
\newtheorem*{theorem*}{Theorem}
\newtheorem{lemma}[theorem]{Lemma}
\newtheorem*{lemma*}{Lemma}
\newtheorem{corollary}[theorem]{Corollary}
\newtheorem{proposition}[theorem]{Proposition}

\theoremstyle{remark}
\newtheorem{remark}[theorem]{Remark}
\newtheorem*{remark*}{Remark}

\theoremstyle{definition}
\newtheorem{definition}[theorem]{Definition}
\newtheorem*{definition*}{Definition}

\numberwithin{equation}{section}

\newcount\quantno
\everydisplay{\quantno=0}\everycr{\quantno=0}
\newcommand\quant{\advance\quantno by1
                      \ifnum\quantno=1\qquad\else\quad\fi\forall }
\newcommand\itemno[1]{(\romannumeral #1)}

\renewcommand\Re{\operatorname{\mathrm{Re}}}
\renewcommand\Im{\operatorname{\mathrm{Im}}}

\newcommand\Dom{\mathrm{Dom}}

\newcommand\rest[1]{\kern-.1em
          \lower.5ex\hbox{$\scriptstyle #1$}\kern.05em}

\newcommand\set[1]{{\left\{#1\right\}}}

\renewcommand\mod[1]{\vert{#1}\vert}
\newcommand\bigmod[1]{\bigl\vert{#1}\bigr|}
\newcommand\Bigmod[1]{\Bigl\vert{#1}\Bigr|}

\newcommand\norm[2]{{\Vert{#1}\Vert_{#2}}}

\newcommand\bignorm[2]{\left.{\bigl\Vert{#1}\bigr\Vert_{#2}}\right.}
\newcommand\bignormto[3]{\left.{\bigl\Vert{#1}\bigr\Vert_{#2}^{#3}}\right.}
\newcommand\Bignorm[2]{\left.{\Bigl\Vert{#1}\Bigr\Vert_{#2}}\right.}

\newcommand\bigopnorm[2]{\big|\!\big|\!\big| {#1} \big|\!\big|\!\big|_{#2}}

\newcommand\prodo[2]{\left\langle#1,#2\right\rangle}

\newcommand\wrt{\,\text{\rm d}}

\newcommand\opL{\operatorname{\mathcal{L}}}

\newcommand\BC{\mathbb{C}}
\newcommand\BD{\mathbb{D}}

\newcommand\BR{\mathbb{R}}

\newcommand\BX{\mathbb{X}}
\newcommand\BZ{\mathbb{Z}}

\newcommand\fra{\mathfrak{a}}  
\newcommand\cB{\mathcal{B}}   
\newcommand\frb{\mathfrak{b}}

\newcommand\cD{\mathcal{D}}  
 
\newcommand\cE{\mathcal{E}}

\newcommand\frg{\mathfrak{g}} 
\newcommand\cH{\mathcal{H}}   
\newcommand\frh{\mathfrak{h}} 
\newcommand\cI{\mathcal{I}}   
 
\newcommand\cJ{\mathcal{J}}
   \newcommand\frk{\mathfrak{k}} 
\newcommand\cL{\mathcal{L}}    
  \newcommand\fM{\mathfrak{M}} \newcommand\frm{\mathfrak{m}}  
    
\newcommand\cO{\mathcal{O}}   \newcommand\fro{\mathfrak{o}} 
\newcommand\cP{\mathcal{P}}   \newcommand\frp{\mathfrak{p}}  

\newcommand\cR{\mathcal{R}}

\newcommand\cU{\mathcal{U}}

  \newcommand\fX{\mathfrak{X}}  
\newcommand\cY{\mathcal{Y}}

\newcommand\astar{\fra^*}
\newcommand\astarc{\fra_{\BC}^*}

\newcommand\al{\alpha}
\newcommand\be{\beta}
\newcommand\ga{\gamma}    
\newcommand\de{\delta}

\newcommand\la{\lambda}   
\newcommand\om{\omega}      
\newcommand\si{\sigma}    
\newcommand\te{\theta}
\newcommand\vp{\varphi}

\newcommand\OV{\overline}
\newcommand\funnyk{k\hbox to 0pt{\hss\phantom{g}}}

\newcommand\lu[1]{L^1(#1)}
\newcommand\lp[1]{L^p(#1)}

\newcommand\ld[1]{L^2(#1)}

\newcommand\ly[1]{L^\infty(#1)}

\newcommand\hu[1]{H^1(#1)}
\newcommand\wthuR[1]{\widetilde H_{\cR}^1(#1)}
\newcommand\huR[1]{H_{\cR}^1(#1)}

\newcommand\huH[1]{H_{\cH}^1(#1)}

\newcommand\huP[1]{H_{\cP}^1(#1)}

\newcommand\Xh[1]{X^k(#1)}

\newcommand\ghu[1]{{\frh}^1(#1)}
\newcommand\gbmo[1]{{\frb\frm\fro}(#1)}

\newcommand\gXum[1]{\fX^{1/2}(#1)}
\newcommand\gXu[1]{{\fX}^1(#1)}
\newcommand\gXh[1]{{\fX}^k(#1)}
\newcommand\gXga[1]{\fX^{\ga}(#1)}
\newcommand\Xxat[1]{\fX_{\mathrm{at}}^k(#1)}

\newcommand\Xnat{\fX_{\mathrm{at}}^k}

\newcommand\wh{\widehat}
\newcommand\wt{\widetilde}
\newcommand\whH{\widehat{\phantom{G}}\hbox to 0pt{\hss $H$}}

\newcommand\emspace{\hbox to 6pt{\hss}}
\newcommand\ds{\displaystyle}

\newcommand\rmi{\hbox{\rm (i)}}
\newcommand\rmii{\hbox{\rm (ii)}}
\newcommand\rmiii{\hbox{\rm (iii)}}
\newcommand\rmiv{\hbox{\rm (iv)}}
\newcommand\rmv{\hbox{\rm (v)}}
\newcommand\rmvi{\hbox{\rm (vi)}}

\newcommand\ir{\int_{-\infty}^{\infty}}

\newcommand\e{\mathrm{e}}

\newcommand\supp{\mathrm{supp}}

\DeclareSymbolFont{EUEX}{U}{euex}{m}{n}

\DeclareSymbolFont{euexlargesymbols}{U}{euex}{m}{n}
\DeclareMathSymbol{\intop}{\mathop}{euexlargesymbols}{"52}
     \def\int{\intop\nolimits}

\DeclareSymbolFont{euexsymbols}     {U}{euex}{m}{n}
\DeclareMathSymbol{\smallint}{\mathop}{euexsymbols}{"52}

\begin{document}

\title[A family of Hardy type spaces]{A family of Hardy type spaces\\ on nondoubling manifolds}  



\thanks{Work partially supported by PRIN 2015 ``Real and complex manifolds: 
geometry, topology and harmonic analysis'' and by the EPSRC Grant EP/P002447/1.
The authors are members of the Gruppo Nazionale per l'Analisi Matematica, 
la Probabilit\`a e le loro Applicazioni (GNAMPA) of the Istituto 
Nazionale di Alta Matematica (INdAM)
}

\author[A. Martini]{Alessio Martini}
\address[Alessio Martini]{
School of Mathematics \\ University of Birmingham\\ Edgbaston \\ Birmingham \\ B15 2TT  \\ United Kingdom}
\email{a.martini@bham.ac.uk}

\author[S. Meda]{Stefano Meda}
\address[Stefano Meda]{Dipartimento di Matematica e Applicazioni \\ Universit\`a di Milano-Bicocca\\
via R.~Cozzi 53\\ I-20125 Milano\\ Italy}
\email{stefano.meda@unimib.it}

\author[M. Vallarino]{Maria Vallarino}
\address[Maria Vallarino]{Dipartimento di Scienze Matematiche ``Giuseppe Luigi Lagrange'', Dipartimento di Eccellenza 2018-2022 \\
Politecnico di Torino\\ corso Duca degli Abruzzi 24\\ 10129 Torino\\ Italy}
\email{maria.vallarino@polito.it}

\begin{abstract}
We introduce a decreasing one-parameter family $\gXga{M}$, $\ga>0$, of Banach subspaces 
of the Hardy--Goldberg space $\ghu{M}$ on certain nondoubling Riemannian 
manifolds with bounded geometry and we investigate their properties. In particular, we prove that 
$\gXum{M}$ agrees with the space
of all functions in $\ghu{M}$ whose Riesz transform is in $\lu{M}$, and we obtain the surprising result that this space
does not admit an atomic decomposition. 
\end{abstract}

\maketitle

\section{Introduction} \label{s:Introduction}

In their seminal paper \cite{FS} C.~Fefferman and E.M.~Stein defined 
the classical Hardy space $\hu{\BR^n}$ as follows:
\begin{equation} \label{f: HuBRn}
\hu{\BR^n}
:= \{f \in \lu{\BR^n}: \bigmod{\nabla (-\Delta)^{-1/2} f} \in \lu{\BR^n}\};
\end{equation} 
here $\nabla$ and $\Delta$ denote the Euclidean gradient and Laplacian, respectively.  
Following up earlier work of D.L.~Burkholder, R.F.~Gundy and M.L.~Silverstein \cite{BGS}, 
Fefferman and Stein obtained several characterisations 
of $\hu{\BR^n}$ in terms of various
maximal operators and square functions, thereby starting the real variable theory of Hardy spaces.
Their analysis was complemented by R.R.~Coifman \cite{Coi}, who showed that  
$\hu{\BR}$ admits an atomic decomposition.  This result was
later extended to higher dimensions by R.~Latter \cite{La}.  

It is natural to speculate whether an analogue of the results of Fefferman--Stein, Coifman and Latter holds in different settings.
In other words, one may ask what is the most appropriate way to define Hardy spaces in settings other than $\BR^n$ and whether different definitions lead to the same spaces. In this paper we will consider this problem on a class of nondoubling Riemannian manifolds.

\medskip

There is a huge literature concerning this question on manifolds or on even more
abstract sorts of spaces and it is virtually impossible to give an account of the main results in the field. Thus, without any pretence of exhaustiveness, we mention just a few contributions, which we consider the most relevant to our discussion.
It is fair to say that most of the results in the literature are concerned with settings where the relevant metric and measure satisfy the doubling condition; while these works do not directly apply to the manifolds considered here, they nevertheless play a paradigmatic role in the development of the subject and it would be impossible to leave them out of our discussion.

In the context of spaces of homogeneous type, Coifman and G.~Weiss \cite{CW} defined an atomic Hardy space which generalises the Euclidean one.
Various maximal function characterisations of such Hardy space have been obtained by a number of authors under additional assumptions on the underlying metric (see, e.g., \cite{Umax,YZ1} and references therein); however, a characterisation in terms of singular integrals, similar to the Euclidean one via Riesz transforms, remains in general a deceptive problem.
A complete characterisation of the Coifman--Weiss Hardy space in terms of maximal functions and Riesz transforms was carried over 
by G.~Folland and Stein \cite{FS} and by M.~Christ and D.~Geller \cite{CG} in the case of stratified groups,
following a deep result of A.~Uchiyama \cite{U}.
Further results in this direction were obtained by J.~Dziuba\'nski and K.~Jotsaroop 
\cite{DJ} in $\BR^n$, but with the Grushin operator playing the role of the Euclidean Laplacian,
and by Dziuba\'nski and J.~Zienkiewicz (see \cite{DZ1,DZ2} and the references therein)
for Schr\"odinger operators with nonnegative potentials satisfying certain additional assumptions.

Within the class of Riemannian manifolds with doubling Riemannian measure, a consequence of works of various
authors \cite{AMR, HLMMY, DKKP} (see also the references therein) is that, under mild geometric
conditions, the Hardy spaces defined in terms of the heat maximal operator and the Poisson maximal operator agree, and 
coincide with an atomic Hardy space defined in terms of appropriate atoms (these are atoms naturally associated
to the Laplace--Beltrami operator and may differ
considerably from those defined by Coifman and Weiss). Furthermore, in this setting it is known that
the Riesz--Hardy space contains the atomic space, but, to the best of our knowledge,
the question whether this inclusion is proper is still open.  
 
\medskip
 
In this paper we consider Riemannian manifolds $M$ with positive injectivity radius, 
Ricci tensor bounded from below and spectral gap.  
Notice that $M$, equipped with the Riemannian distance, is \emph{not}
a space of homogeneous type in the sense of Coifman and Weiss, for 
the doubling condition fails for large balls.  The corresponding theory of atomic Hardy type 
spaces has been developed only quite recently \cite{Io,CMM1,T,MMV2}, and it differs remarkably from that of $\hu{\BR^n}$.  
A related nondoubling setting where interesting results concerning aspects of this programme 
have been developed has been considered in \cite{V,MOV}.

The final outcome of our research, which will be described in detail in the present paper and in other
forthcoming papers, is that different definitions of Hardy spaces (atomic, via Riesz transform, via maximal operators)
on certain manifolds with exponential volume growth may very well lead to different spaces.   
Related interesting partial results are \cite[Corollary~6.3]{A1} and \cite{Lo}.  

\medskip

Our work is inspired by a series of papers \cite{MMV1,MMV2}, by the Ph.D.\ thesis \cite{Vo} and by the recent work \cite{CM} on graphs.
Specifically, in \cite{MMV1} the Authors introduced a sequence $\Xh{M}$ of strictly decreasing Banach spaces,
which are isometric copies of the Hardy type space $\hu{M}$,
introduced by A.~Carbonaro, G.~Mauceri and Meda in \cite{CMM1}.  
This space differs from the classical Hardy space of Coifman--Weiss \cite{CW}.  

Volpi \cite{Vo} modified this construction by letting the Hardy--Goldberg type
space $\ghu{M}$, introduced by M.~Taylor in \cite{T} and further generalised by Meda and Volpi \cite{MVo},
play the role of the space $\hu{M}$ of Carbonaro, Mauceri and Meda.   
The resulting sequence of spaces is named $\gXh{M}$, instead of $\Xh{M}$.  Of course
$\gXh{M} \supseteq \Xh{M}$, for $\ghu{M}$ properly contains $\hu{M}$.  
It may be worth recalling that $\ghu{M}$ is the analogue on $M$ of the classical space 
$\ghu{\BR^n}$ introduced by D.~Goldberg in \cite{Go}
and further investigated on specific measure metric spaces in various papers, including \cite{HMY,YZ1,YZ2,BDL}
(see also the references therein).

In this paper we take a step further, and consider 
a one-parameter family of spaces $\gXga{M}$, where $\ga$ is a positive real number, which agree with those introduced
in \cite{Vo} when $\ga$ is a positive integer.   Specifically, the space $\gXga{M}$
is just $\cU^\ga \big[ \ghu{M} \big]$, where $\cU = \cL(\cI+ \cL)^{-1}$ and $\cL$ is the (positive) Laplace--Beltrami
operator on $M$.
It is not hard to see that $\cU$ is injective on $\lu{M}$, hence so is $\cU^\ga$.   
The space $\gXga{M}$ is endowed with the norm that makes $\cU^\ga$ an isometry between $\ghu{M}$
and $\gXga{M}$, i.e.,  $\bignorm{f}{\gXga{M}} := \bignorm{\cU^{-\ga} f}{\ghu{M}}$.

The idea of considering noninteger values of $\ga$ is taken
from \cite{CM}, where an analogue of $\gXga{M}$ is defined on certain graphs
with exponential volume growth.  However,
the case of Riemannian manifolds we consider here
requires substantial refinements of the theory developed in \cite{CM}.

The spaces $\gXga{M}$ play a central role in our analysis
of Hardy type spaces.  We prove that $\gXga{M}$ is a decreasing family of Banach spaces, each of which interpolates with $L^2(M)$. We also show that the imaginary powers of the Laplace--Beltrami operator $\cL$ are bounded from $\gXga{M}$ to $\ghu{M}$ for all $\ga > 0$, thus providing an endpoint counterpart to their $L^p$ boundedness for $p \in (1,\infty)$; notice that the imaginary powers of $\cL$ may not be bounded from $\ghu{M}$ to $L^1(M)$ \cite{MMV4}. 

We also prove that $\gXga{M}$ does not admit an atomic decomposition when $\ga$ is not an integer, at least in the
case of symmetric spaces of the noncompact type and real rank one. More precisely, we show that the space of compactly supported elements of $\gXga{M}$ is not dense in $\gXga{M}$. 

The extension to noninteger values of the parameter $\ga$ is \emph{a posteriori} motivated by one of our main results, which
states that 
\begin{equation} \label{f: equiv}
\gXum{M}
= \{f \in \ghu{M}: \mod{\cR f} \in \lu{M}\}.  
\end{equation}
Here $\cR$ denotes the \emph{Riesz transform} $\nabla \cL^{-1/2}$
on $M$, and $\nabla$ is the Riemannian gradient.   
Notice that we do not prove here that the Riesz--Hardy space $\huR{M}$, defined by 
\begin{equation}\label{f: rieszhardy}
\huR{M}
:= \{f \in \lu{M}: \mod{\cR f} \in \lu{M}\},
\end{equation}
agrees with $\gXum{M}$.  The proof of this equivalence requires \eqref{f: equiv} together
with more sophisticated real variable methods, and will be given in \cite{MVe}. In conjunction with the results of the present paper, this equivalence implies the perhaps surprising result that $\huR{M}$ does not admit in general an atomic decomposition.

The relations between the spaces $\gXga{M}$ for different values of $\gamma>0$ and the Hardy spaces $\huH{M}$ and $\huP{M}$ defined in terms of the 
heat and the Poisson maximal operators will be discussed in detail in \cite{MMV}, yielding another possibly surprising result: the spaces $\huR{M}$, $\huH{M}$ and $\huP{M}$ may all differ in this context.

\bigskip

We shall use the ``variable constant convention'', and denote by $C$,
possibly with sub- or superscripts, a constant that may vary from place to 
place and may depend on any factor quantified (implicitly or explicitly) 
before its occurrence, but not on factors quantified afterwards.

\section{Background on Hardy type spaces}
\label{s: Background material}

Let $M$ denote a connected, complete $n$-dimensional Riemannian manifold
of infinite volume with Riemannian measure $\mu$.  
Denote 
 by $\cL$ the positive Laplace--Beltrami operator 
on $M$, by $b$ the bottom of the $\ld{M}$ spectrum of $\cL$,
and set $\be = \limsup_{r\to\infty} \bigl[\log\mu\bigl(B_r(o)\bigr)\bigr]/(2r)$, 
where $o$ is any reference point of $M$ and $B_r(o)$ denotes the ball centred at $o$ of radius $r$. 
By a result of Brooks, $b\leq \be^2$ \cite{Br}.

We denote by $\cB$ the family of all geodesic balls on $M$.
For each $B$ in $\cB$ we denote by $c_B$ and $r_B$
the centre and the radius of $B$ respectively.  
Furthermore, for each positive number $\lambda$, we denote by $\lambda \, B$ the
ball with centre $c_B$ and radius $\lambda \, r_B$.
For each \emph{scale parameter} $s$ in $\BR^+$, 
we denote by $\cB_s$ the family of all
balls $B$ in $\cB$ such that $r_B \leq s$.  

In this paper we make the following \textbf{assumptions on the geometry of the manifold}:
\begin{enumerate}
\item[\itemno1]
the injectivity radius of $M$ is positive;
\item[\itemno2]
the Ricci tensor is bounded from below; 
\item[\itemno3]
$M$ has spectral gap, to wit $b>0$. 
\end{enumerate}
We emphasize the fact that some of the results in this paper hold under less stringent
assumptions.  For instance, Theorem~\ref{t: bound semigroup O} below requires that $M$ satisfies 
the local doubling condition and supports a local scaled $L^2$-Poincar\'e inequality, which are implied 
by \rmii\ above, but they do not require \rmi\ and \rmiii.  We believe that it is not
worth keeping track of the minimal assumptions under which each of the results below holds, and
assume throughout that $M$ satisfies \rmi--\rmiii\ above.  

It is well known that for manifolds satisfying \rmi--\rmiii\ above the following properties hold:
\begin{enumerate}
\item[(a)]
there are positive constants 
$\al$ and $C$ such that
\begin{equation} \label{f: volume growth} 
\mu(B)
\leq C \, r_B^{\al} \, \e^{2\be \, r_B}
\quant B \in\cB\setminus \cB_1,
\end{equation}
where $\beta$ is the constant defined at the beginning of this section;
\item[(b)]
(see \cite[Remark~2.3]{MMV2}) there exists a positive 
constant $C$ such that
\begin{equation} \label{f: lower bound balls}
C^{-1}\,r_B^n
\leq \mu(B)
\leq C\,r_B^n
\quant B\in\cB_1;
\end{equation}
\item[({c})] as a consequence of (a) and (b) 
the measure $\mu$ is \emph{locally doubling}, 
i.e., for every $s>0$ there exists a constant~$D_s$ such that 
$$
\mu(2B)\le D_s\ \mu(B) \qquad
\forall B\in \cB_s;
$$
\item[(d)]
$M$ possesses a \emph{local scaled $L^2$-Poincar\'e inequality}, i.e.,  
for each $R<\infty$ there exists a constant $C$, depending on $R$, such that 
\begin{equation} \label{f: poincare}
\int_B \bigmod{f-f_B}^2 \wrt \mu
\leq C \, r_B^2 \, \int_B \bigmod{\nabla f}^2 \wrt \mu\,,
\end{equation}
for all balls $B$ in $\cB_R$;
\item[(e)]
the heat semigroup $\{\cH_t\}$ is \emph{ultracontractive}, in the sense 
that $\cH_t := \e^{-t\cL}$ maps $\lu{M}$ into $\ld{M}$ and satisfies the following estimate \cite[Section 7.5]{Gr1}: 
$$
\bigopnorm{\cH_t}{1;2}\leq C \e^{-b t} \,t^{-n/4}\,(1+t)^{\frac{n}{4}}\qquad \forall t\in \BR^+;
$$
it follows by interpolation that, for every $p\in (1,2]$,
\begin{equation}\label{f: heat1p}
\bigopnorm{\cH_t}{1;p}\leq C [\e^{-b t} \,t^{-n/4}\,(1+t)^{\frac{n}{4}}]^{2/{p'}}\qquad \forall t\in \BR^+,
\end{equation}
and in particular
\begin{equation} \label{f: ultracontractivity}
\lim_{t\to\infty} \, \bignorm{\cH_t f}{L^p} = 0
\quant f \in \lu{M};
\end{equation} 
\item[(f)] the Cheeger isoperimetric constant of $M$ is positive \cite[Theorem 9.5]{CMM1}, and consequently the following Sobolev type inequality holds:
\begin{equation} \label{f: FF}
\bignorm{f}{L^1} \leq C \bignorm{ |\nabla f| }{L^1}
\end{equation}
for all $f \in C^\infty_c(M)$ \cite[Theorem V.2.1]{Ch}.
\end{enumerate}

Next, we introduce the local Hardy space $\ghu{M}$.  

\begin{definition} \label{d: atom}
Suppose that $p$ is in $(1,\infty]$ and let $p'$ be the index conjugate to $p$. 
A \emph{standard $p$-atom} 
is a function $a$ in $L^1(M)$ supported in a ball $B$ in $\cB_1$ 
satisfying the following conditions:
\begin{enumerate}
\item[\itemno1] \emph{size condition}: $\norm{a}{L^p}  \leq \mu (B)^{-1/p'}$;
\item[\itemno2] \emph{cancellation condition}:  
$\ds \int_B a \wrt \mu  = 0$. 
\end{enumerate}
A \emph{global $p$-atom} is a function $a$
in $L^1(M)$ supported in a ball $B$ of radius \emph{exactly equal to} $1$ 
satisfying the size condition above (but possibly not the cancellation condition).
Standard and global $p$-atoms will be referred to simply as $p$-\emph{atoms}.
\end{definition}

\begin{definition} \label{d: Goldberg}
The \emph{local atomic Hardy space} $\frh^{1,p}({M})$ is the 
space of all functions~$f$ in $L^1(M)$
that admit a decomposition of the form
\begin{equation} \label{f: decomposition}
f = \sum_{j=1}^\infty \lambda_j \, a_j,
\end{equation}
where the $a_j$'s are $p$-atoms 
and $\sum_{j=1}^\infty \mod{\lambda_j} < \infty$.
The norm $\norm{f}{\frh^{1,p}}$
of $f$ is the infimum of $\sum_{j=1}^\infty \mod{\lambda_j}$
over all decompositions (\ref{f: decomposition}) of $f$.
\end{definition}

This space was introduced in even greater generality by Volpi \cite{Vo}, who extended previous 
work of Goldberg \cite{Go} and Taylor \cite{T}, and then further generalised in \cite{MVo}.
Goldberg treated the Euclidean case, while Taylor worked on Riemannian manifolds with strongly bounded geometry and considered only $\infty$-atoms (more precisely, ions); Volpi worked in a much more abstract setting, which covers the case where $M$
is a Riemannian manifold with Ricci curvature bounded from below (see also \cite{MVo}
for more on this).  In particular, Volpi proved that 
$\frh^{1,p}(M)$ is independent of $p$; henceforth, the space
$\frh^{1,2}(M)$ will be denoted simply by $\ghu{M}$, and $2$-atoms in $\ghu{M}$ will also be called $\ghu{M}$-atoms.

The choice of $1$ as a ``scale parameter'' for the radii of balls in Definition \ref{d: atom} is completely arbitrary, and replacing it with any other positive number would lead to the definition of the the same space $\ghu{M}$, with equivalent norms.
Indeed, a slight modification of \cite[Lemma 2]{MVo} 
shows that, for all $p\in (1,\infty]$, there exists a constant $C$ 
such that, for every function $f$ in $L^p(M)$ supported in a ball 
$B\in \cB\setminus \cB_1$, the function $f$ is in $\mathfrak h^1(M)$ and 
\begin{equation}\label{f: normah1funzioneL2}
\|f\|_{\frh^1}\leq C\mu(B)^{1/p'}\,\|f\|_{L^p}\,.
\end{equation}

An important feature of $\ghu{M}$ lies in its interpolation properties with $L^p$ spaces. In particular, for every $\theta$ in $(0,1)$, the complex interpolation space 
$\big(\ghu{M},\ld{M}\big)_{[\theta]}$ is $L^{2/(2-\theta)}(M)$ (see \cite[Theorem~5]{MVo}).  

We shall repeatedly use the following proposition, whose proof is a slight modification of \cite[Theorem 6]{MVo}. 

\begin{proposition}\label{p: uniflim}
If $T$ is an $\ghu{M}$-valued linear operator defined on $\ghu{M}$-atoms such that  
$$
\sup\{\|Ta\|_{\frh^1}: a \text{ $\ghu{M}$-atom}\}<\infty\,,
$$
then $T$ admits a unique bounded extension from $\mathfrak{h}^1(M)$ to $\mathfrak{h}^1(M)$. 
\end{proposition}

The definition of the space $\ghu{M}$ is similar to that of the atomic Hardy space $\hu{M}$, introduced by 
Carbonaro, Mauceri and Meda \cite{CMM1,CMM2}, the only difference
being that atoms in $\hu{M}$ are just standard atoms in $\ghu{M}$, and there are no global atoms.  
As a consequence, functions in $\hu{M}$ have vanishing integral, a property not enjoyed by all the 
functions in $\ghu{M}$.
Thus, trivially, $\hu{M}$ is properly and continuously contained in $\ghu{M}$.  

We now introduce the space $\gbmo{M}$.  
Suppose that $q$ is in $[1,\infty)$.  
For each locally integrable function $g$ define the 
\emph{local sharp maximal function} $g^{\sharp,q}$ by
$$
g^{\sharp,q}(x)
= \sup_{B \in \cB_{1}(x)} \Bigl(\frac{1}{\mu(B)}
\int_B \mod{g-g_B}^q \wrt\mu \Bigr)^{1/q}
\quant x \in M,
$$
where $g_B$ denotes the average of $f$ over $B$ and $\cB_{1}(x)$ denotes the family
of all balls in $\cB_{1}$ centred at the point $x$. 
Define the \emph{modified local sharp maximal function} $N^q(g)$ by
$$
N^q(g)(x)
:= g^{\sharp,q}(x) +  \Bigl[\frac{1}{\mu(B_1(x))} 
     \int_{B_1(x)} \mod{g}^q \wrt\mu \Bigr]^{1/q}
\quant x \in M\,.
$$
Denote by $\mathfrak{bmo}^q(M)$ the space of all locally integrable 
functions~$g$ such that $N^q(g)$ is in $\ly{M}$, endowed with the norm 
$$
\norm{g}{\mathfrak{bmo}^q}
= \norm{N^q(g)}{L^{\infty}}.
$$

In \cite{MVo} it is proved that 
the space $\mathfrak{bmo}^q(M)$ does not depend on the parameter $q$,
as long as $q$ is in $[1,\infty)$.  Henceforth, we shall denote this space
by $\gbmo{M}$, endowed with the norm $\mathfrak{bmo}^2$. 
Moreover, the space $\gbmo{M}$ may be identified with the dual of $\ghu{M}$ (see \cite[Theorem 2]{MVo}). 
More precisely, for every function $g\in \gbmo{M}$, the linear functional $F_g$, defined on every $\ghu{M}$-atom $a$ by
\begin{equation}\label{f: dualita} 
F_g(a)=\int_M a\,g \wrt \mu,
\end{equation}
extends to a bounded linear functional on $\ghu{M}$. 
Conversely, for every functional $F\in(\ghu{M})'$ 
there exists a function $g\in \gbmo{M}$ such that $F=F_g$. Moreover, there exists a positive constant $C$ such that 
\begin{equation}\label{f: dualita_norme}
C^{-1} \|g\|_{\mathfrak{bmo}}\leq  \|F\|_{(\mathfrak{h}^1)'}\leq C\,\|g\|_{\mathfrak{bmo}}\,.
\end{equation}

\section{The heat semigroup and the operator \texorpdfstring{$\cU$}{U} on \texorpdfstring{$\ghu{M}$}{h1(M)}}
\label{s: The heat semigroup on ghuM}

The theory of Hardy type spaces that we shall describe in 
Section~\ref{s: A one parameter family of Hardy-type spaces} requires 
the boundedness on $\ghu{M}$ of various functions of the Laplace--Beltrami operator $\cL$,
including the heat semigroup and the operator $\cU$ defined below.  These will be established in 
Subsections~\ref{subs: The heat semigroup on ghuM} and~\ref{subs: The operator cU on ghuM}, respectively.  

\subsection{The heat semigroup on \texorpdfstring{$\ghu{M}$}{h1(M)}} \label{subs: The heat semigroup on ghuM} 

It is well known that $\big\{\cH_t\big\}$ is a Markovian semigroup.  In particular,
it is contractive on $\lu{M}$, hence from $\ghu{M}$ to $\lu{M}$, for $\ghu{M}$
is continuously imbedded in $\lu{M}$ with norm $\leq 1$.  
In this section we discuss the boundedness of the heat semigroup $\{\cH_t\}$
on the local Hardy space $\ghu{M}$.  

A well known result obtained independently by Grigor'yan and 
Saloff-Coste~\cite[Theorem~5.5.1]{SC}
says that the conjunction of the local doubling condition and the local Poincar\'e inequality is 
equivalent to a local Harnack inequality for positive solutions to the
heat equation.  In particular, for all $R>0$ there exists a constant $C$ such that for all
balls $B=B(c_B,r_B)$, with $r_B<R$, and for any smooth positive solution $u$ of 
$(\partial_t+\cL)u=0$ in the cylinder $Q := \big(s-r_B^2,s\big) \times B$, the
following inequality holds:
\begin{equation} \label{f: harnack}
\sup_{Q_-} u
\leq C \, \inf_{Q_+} u, 
\end{equation}
where 
$\ds Q_- := \Big(s-\frac{3r_B^2}{4},s-\frac{r_B^2}{2}\big) \times B(c_B,r_B/2)$
and
$\ds Q_+ := \Big(s-\frac{r_B^2}{4},s\Big) \times B(c_B,r_B/2)$.  

For $\om$ in $[0,\pi]$, we denote by $S_\om$ the half line $(0,\infty)$ if $\om =0$,
and the sector $\big\{z \in \BC: z \neq 0, \, \hbox{and $\bigmod{\arg z} < \om$} \big\}$
if $\om>0$.  Recall that, given a number $\om$ in $[0,\pi)$,   
an operator~$A$ on a Banach space $\cY$ is \emph{sectorial} of angle $\om$ if 
\begin{enumerate}
\item[\itemno1]
the spectrum of $A$ is contained in the closed sector $\OV{S}_\om$;
\item[\itemno2]
the following \emph{resolvent estimate} holds:
$$
\sup_{\la \in \BC\setminus \OV{S}_{\om'}} \bigopnorm{\la \, (\la- A)^{-1}}{\cY} 
< \infty
\quant \om' \in (\om,\pi).  
$$
\end{enumerate}
Observe that condition \rmii\ above may be reformulated as follows:
\begin{equation} \label{f: equiv ref sectoriality}
\sup_{\la \in {S}_{\om'}} \bigopnorm{\la \, (\la+ A)^{-1}}{\cY} 
< \infty
\quant \om' \in [0,\pi-\om).  
\end{equation}
The theory of Hardy type spaces that we shall develop in 
Section~\ref{s: A one parameter family of Hardy-type spaces} hinges on the uniform boundedness of 
the heat semigroup on $\ghu{M}$.  This fact, together with some related estimates, 
will be proved in the next theorem.  A result similar to Theorem~\ref{t: bound semigroup O}~\rmi\ below,
but in a different setting, may be found in \cite{DW}.  

\begin{theorem} \label{t: bound semigroup O}
The following hold:
\begin{enumerate}
\item[\itemno1]
$\{\cH_t\}$ is a uniformly bounded $C_0$ semigroup on $\ghu{M}$;
\item[\itemno2]
$\cL$ is a sectorial operator of angle $\pi/2$ on $\ghu{M}$;
\item[\itemno3]
$\sup_{\la>0} \, \, \bigopnorm{\la \, (\la + \cL)^{-1}}{\mathfrak h^1} < \infty$ and
$\sup_{\la>0} \bigopnorm{\cL \, (\la + \cL)^{-1}}{\mathfrak h^1} < \infty$.
\end{enumerate}
\end{theorem}

\begin{proof}
First we prove \rmi. 
We shall preliminarily show that 
\begin{equation} \label{f: est heat semigroup}
\sup_{t>0} \, \sup \, \bigmod{\prodo{\cH_t a}{g}}
< \infty,
\end{equation}
where the inner supremum is taken over all $\ghu{M}$-atoms $a$ and all $\gbmo{M}$-functions~$g$
with $\bignorm{g}{\mathfrak{bmo}} \leq 1$. In light of \eqref{f: dualita_norme} and Proposition \ref{p: uniflim}, estimate \eqref{f: est heat semigroup} implies
the uniform boundedness of $\{\cH_t\}$ on $\ghu{M}$.

To prove \eqref{f: est heat semigroup}, let $a$ be an $\ghu{M}$-atom supported in a ball $B=B(c_B,r_B)$, with $r_B\leq 1$. 
Denote by $\fM$ a $1$-discretisation of $M$
and, for each~$z$ in $\fM$, denote by $B_z$ the ball with centre $z$ and radius $1$.  
It is a well known fact (see, for instance, \cite{MVo}) that the cover $\{B_z: z \in \fM\}$ has the
finite overlapping property.  Denote by $\{\psi_z: z \in \fM\}$ a partition of
unity subordinate to that cover.  Clearly, at least formally, 
$$
\begin{aligned}
\bigmod{\prodo{\cH_t a}{g}}
\leq \sum_{z \in \fM} \, \int_{B_z} \psi_z(x) \, \bigmod{\cH_t a(x)} \, \bigmod{g(x)}\wrt \mu(x) 
\leq \sum_{z \in \fM} \, \bignorm{\cH_t a}{\ld{B_z}} \, \bignorm{g}{\ld{B_z}}.  
\end{aligned}
$$
We have used the Cauchy--Schwarz inequality in the last inequality above.  
Notice that 
$$
\begin{aligned}
\bignorm{g}{\ld{B_z}}
=     \mu(B_z)^{1/2} \, \Big[ \frac{1}{\mu(B_z)} \int_{B_z} \mod{g}^2 \wrt \mu\Big]^{1/2} 
\leq  \mu(B_z)^{1/2} \bignorm{g}{\mathfrak{bmo}}.  
\end{aligned}
$$
Furthermore, if $h_t$ denotes the heat kernel, then
$$
\begin{aligned}
\bignorm{\cH_t a}{\ld{B_z}}
&  =    \Big[ \int_{B_z} \wrt \mu(x) \, \Bigmod{\int_B h_t(x,y) \, a(y) \wrt \mu(y)}^2 
          \Big]^{1/2} \\
&  \leq \Big[ \int_{B_z} \wrt \mu(x) \, \int_B h_t(x,y)^2 \, \bignormto{a}{\ld{B}}{2}  \wrt \mu(y)
          \Big]^{1/2} \\
&  \leq \Big[ \int_{B_z} \wrt \mu(x) \, \frac{1}{\mu(B)} \, \int_B h_t(x,y)^2 \wrt \mu(y)
          \Big]^{1/2};
\end{aligned}
$$
we have used Schwarz's inequality in the first inequality and the size condition on $a$ in the second.  
Clearly for each $x$ in $M$
$$
\begin{aligned}
\frac{1}{\mu(B)} \, \int_B h_t(x,y)^2 \wrt \mu(y)
\leq \sup_{y\in B} \, h_t(x,y)^2   
\leq C\, \inf_{y\in B} \, h_{t+r_B^2}(x,y)^2   
\leq C\, h_{t+r_B^2}(x,c_B)^2;  
\end{aligned}
$$
we have used Harnack's inequality in the second inequality above.  
Thus,
$$
\begin{aligned}
\bignorm{\cH_t a}{\ld{B_z}}
\leq C\, \Big[\int_{B_z}  \, h_{t+r_B^2}(x,c_B)^2 \wrt \mu(x) \Big]^{1/2},
\end{aligned}
$$
which, by Harnack's inequality, is dominated by $C\, \mu(B_z)^{1/2}  \, h_{t+r_B^2+1}(z,c_B)$.
By combining the preceding estimates with the finite overlapping property of the family of 
balls $\{B_z: z \in \fM\}$, we see that 
$$
\begin{aligned}
\bigmod{\prodo{\cH_t a}{g}}
& \leq C\, \bignorm{g}{\mathfrak{bmo}}\, \sum_{z \in \fM} \, \mu(B_z) \, h_{t+r_B^2+1}(z,c_B) \\
& \leq C\, \bignorm{g}{\mathfrak{bmo}   }\, \sum_{z \in \fM} \,
                \int_{B_z} h_{t+r_B^2+2}(x,c_B) \wrt \mu(x) \\ 
& \leq C\, \bignorm{g}{ \mathfrak{bmo}}\, \int_{M} h_{t+r_B^2+2}(x,c_B) \wrt \mu(x) \\ 
& =    C\, \bignorm{g}{\mathfrak{bmo} },
\end{aligned}
$$
as required.  The equality above follows from the Markovianity of the heat semigroup.  

In order to conclude the proof of \rmi, it remains to show that $\{\cH_t\}$ is
strongly continuous on $\ghu{M}$.  Notice that it suffices to prove that $\bignorm{\cH_ta-a}{\mathfrak h^1} \to 0$
as $t\to 0^+$ for every $\ghu{M}$-atom $a$.  Indeed, suppose that this holds, and assume that $\ds f= \sum_j c_j \, a_j$.  Then
$$
\cH_t \Big(\sum_{j=1}^\infty\, c_j \, a_j\Big) - \sum_{j=1}^\infty\, c_j \, a_j 
= \sum_{j=1}^\infty\, c_j \big(\cH_t a_j - a_j\big), 
$$
because we already know that $\{\cH_t\}$ is bounded on $\ghu{M}$.
Hence
$$
\bignorm{\cH_tf-f}{\mathfrak h^1} 
\leq \sum_{j=1}^\infty \, \bigmod{c_j} \, \bignorm{\cH_ta_j-a_j}{\mathfrak h^1}
\to 0
$$
as $t\to 0^+$ by the Lebesgue dominated convergence theorem (we have already proved that
$\{\cH_t\}$ is uniformly bounded, so that 
$\bignorm{\cH_ta_j-a_j}{\mathfrak h^1} \leq 1+\sup_{t>0} \, \bigopnorm{\cH_t}{\mathfrak h^1} $).  

Now we show that, if $a$ is an $\ghu{M}$-atom, then $\bignorm{\cH_ta-a}{\mathfrak h^1} \to 0$.  
Without loss of generality, we may assume that the support of $a$ is contained in a ball
$B$ centred at $o$ with radius $R\leq 1$.  Let $\cD=\sqrt{\cL-b}$, and observe that, by the spectral 
theorem, $s\mapsto \cos(s\cD) a$ is an $L^2$-valued continuous function on $\BR$, and that
it takes value $a$ at $0$.  At least on $\ld{M}$,
$$
\cH_t a 
= \e^{-bt}\, \ir h_t^\BR(s) \, \cos(s\cD)a \wrt s,
$$ 
where $h_t^\BR(s) = (4\pi t)^{-1/2} \, \e^{-s^2/(4t)}$ is the one-dimensional Euclidean heat kernel.  
Write $\ds 1 = \om_0 + \sum_{k=1}^\infty \, \om_k$, where $\supp \,\om_0 \subseteq [-1,1]$,
and $\supp \,\om_k \subseteq [-k-1,-k]\cup [k,k+1]$.   
We are led to consider the integrals
$$
I_t^k
:= \e^{-bt} \, \ir h_t^\BR(s) \, \om_k(s) \, \cos(s\cD)a \wrt s\,,
$$
where $k$ is a nonnnegative integer.

Consider first the case where $k>0$. Clearly 
$
\bignorm{\cos(s\cD)a}{L^2} 
\leq \bignorm{a}{L^2} 
\leq \mu(B)^{-1/2}, 
$
and moreover $\supp(\cos(s\cD)a) \subseteq B_{k+2}(o)$ by finite propagation speed; hence, 
by \eqref{f: normah1funzioneL2} and \eqref{f: volume growth}, we conclude that
$$
\bignorm{\cos(s\cD)a}{\mathfrak h^1} 
\leq C \, \frac{\mu\big(B_{k+2}(o)\big)^{1/2}}{\mu(B)^{1/2}} 
\leq C \, k^\nu \, \e^{\beta k}
\quant s: k \leq \mod{s}<k+1,
$$
where $\nu$ is an appropriate nonnegative number.  Remember that $a$ is fixed in this argument, so that 
we do not care about the factor $\mu(B)^{-1/2}$.  Therefore
$$
\bignorm{I_t^k}{\mathfrak h^1}
\leq C \, k^\nu \, \e^{\beta k}\, \int_k^{k+1}  h_t^\BR(s) \wrt s
\leq C \, t^{-1/2}\, k^\nu \, \e^{\beta k-k^2/(4t)}.  
$$
It is straightforward to check that $t^{-1/2}\, k^\nu \, \e^{\beta k-k^2/(4t)} \leq \e^{-c/t-ck^2}$
for some suitable constant $c$ and all $t$ small enough, uniformly in $k$.  Thus, 
$$
\Bignorm{\sum_{k=1}^\infty I_t^k}{\mathfrak h^1}
\leq C\, \sum_{k=1}^\infty \, \e^{-c/t-ck^2}
$$
which tends to $0$ as $t\to 0^+$.  

Thus, it remains to estimate $\bignorm{I_t^0-a}{\mathfrak h^1}$.  
Notice that $t\mapsto I_t^0 -a$ is an $L^2$-valued
continuous function on $\BR$.  
First, we show that $I_t^0-a$ is weakly convergent to $0$ in $\ld{M}$.  Indeed, 
$(a,\psi)_{\ld{M}} = \prodo{(\cos(\cdot \cD)a,\psi)_{\ld{M}}}{\de_{0}}_{\BR}$, 
where $\prodo{\cdot}{\cdot}_{\BR}$ denotes the pairing between
measures and continuous functions on $\BR$.  Thus, if $m$ denotes the Lebesgue measure on $\BR$,
$$
\begin{aligned}
\big(I_t^0 - a, \psi\big)_{\ld{M}}
& = \e^{-bt} \, \ir h_t^\BR(s) \, \om_0(s) \, \big(\cos(s\cD)a,\psi)_{L^2} \wrt s - (a,\psi)_{L^2} \\
& =  \prodo{\big(\cos(\cdot \cD)a,  \psi\big)_{L^2}}{\e^{-bt} \,h_t^\BR \, \om_0 \wrt m\, - \de_{0}}_{\BR},
\end{aligned}  
$$
which tends to $0$ as $t\to 0^+$, because the measure ${\e^{-bt} \,h_t^\BR \, \om_0 \wrt m\, - \de_{0}}$  
is weakly convergent to $0$.  Furthermore, 
$$
\begin{aligned}
\bignorm{I_t^0}{L^2}
\leq \e^{-bt} \, \ir h_t^\BR(s) \, \om_0(s) \, \bignorm{\cos(s\cD)a}{\ld{M}} \wrt s \,  
\leq \bignorm{a}{L^2},
\end{aligned}
$$ 
so that $\limsup_{t\to 0^+} \, \bignorm{I_t^0}{L^2} \leq \bignorm{a}{L^2}$.
By a well known result \cite[Proposition~3.32]{B}, $I_t^0 \to a$ strongly in $\ld{M}$.  

Note that the support of $I_t^0$ is contained in $B_2(o)$.  Hence, by \eqref{f: normah1funzioneL2},
$$
\bignorm{I_t^0-a}{\mathfrak h^1}
\leq C\,\mu\big(B_2(o)\big)^{1/2} \, \bignorm{I_t^0-a}{L^2}
\to 0
$$
as $t \to 0^+$, as required to conclude the proof of the strong continuity of $\{\cH_t\}$ on $\ghu{M}$.

As explained in \cite[Section 2.1.1, p.\ 24]{Haa}, \rmii\ is an immediate consequence of \rmi\ and the Hille--Yosida theorem.

The first statement in \rmiii\ follows directly from the sectoriality of $\cL$ proved in \rmii.
To prove the last statement in \rmiii\ observe that, by spectral theory,  
$$
\cL \, (\la+ \cL)^{-1} = \cI - \la\, (\la+\cL)^{-1},
$$
at least on $\ld{M}$.  The required estimate follows from this and the first
statement.  
\end{proof}

\begin{remark}\label{r: heat semigroup unbounded on huM}
It is worth pointing out that the heat semigroup is \emph{not} uniformly bounded
on $\hu{\BD}$, where $\BD$ denotes the hyperbolic disk.  Indeed, arguing as in \cite[Section 2.7]{Ce}, 
one can show that there exists a positive constant $c$ such that 
$$
\bigopnorm{\cH_t}{\hu{\BD}}
\geq c \, (1+t)
\quant t \in \BR^+.  
$$
We omit the proof, because the result above is not essential for the
theory of Hardy type spaces developed in this paper and the details are somewhat long and intricate.  
This motivates the introduction of the spaces $\gXga{M}$ in 
Section \ref{s: A one parameter family of Hardy-type spaces} 
and explains why we do not base our analysis on the spaces $\Xh{M}$
introduced in \cite{MMV1}.
\end{remark}

\subsection{The operator \texorpdfstring{$\cU$}{U} on \texorpdfstring{$\ghu{M}$}{h1(M)}} \label{subs: The operator cU on ghuM}
A central role in what follows will be played by the family 
$\big\{\cU_\si := \cL\,(\si+\cL)^{-1}$: $\si>0\big\}$ of (spectrally defined) operators. 

Clearly $\cU_\si$ is bounded on $\ld{M}$, by the spectral theorem.  
A straightforward consequence of the fact that
$\cL$ generates the contraction semigroup $\{\cH_t\}$ on $\lp{M}$
for every $p$ in $[1,\infty]$ is that $\cU_\si$ extends to a bounded operator on $\lp{M}$
for all such values of $p$.  Furthermore, for each $\la>0$ the operator $\cU_\si$ is an
isomorphism of $\lp{M}$, $1<p\leq 2$ (because $b>0$, hence the bottom of the $\lp{M}$
spectrum of $\cL$ is positive), and it is injective on $\lu{M}$ \cite[Proposition~2.4]{MMV1}.
We shall often write $\cU$ instead of $\cU_1$.

The sectoriality of $\cL$ on $\ghu{M}$ given by Theorem \ref{t: bound semigroup O} implies the following properties of the operators $\cU_\si$ and their fractional powers.

\begin{proposition} \label{p: consequences of sectoriality}
Assume that $\si>0$.
\begin{enumerate}
\item[\itemno1]
$\cU_\si$ is an injective bounded sectorial operator of angle $\pi/2$ on $\ghu{M}$. 
\item[\itemno2]
$\cU_\si^\ga$ is injective and bounded on $\ghu{M}$ for all $\ga \in \BC$ with $\Re\ga > 0$.
\end{enumerate}  
\end{proposition}

\begin{proof}
First we prove \rmi. Since $\cU_\si$ is injective on $\lu{M}$ it is also injective on $\ghu{M}$.  
By Theorem~\ref{t: bound semigroup O}~\rmii, $\cL$ is a sectorial operator of angle $\pi/2$ on $\ghu{M}$.  
By \cite[Proposition~2.1.1~(f)]{Haa} so is $\cU_\sigma$, and moreover the boundedness of $\cU_\sigma$ follows from Theorem~\ref{t: bound semigroup O}~\rmiii.

Property \rmii\ immediately follows from \rmi\ and \cite[Proposition 3.1.1]{Haa}.
\end{proof}

\begin{remark}\label{rem: imaginarypowers}
The condition $\Re\ga > 0$ in Proposition \ref{p: consequences of sectoriality} \rmii\ cannot be relaxed in general,
for the operators $\cU_\si^{iu}$, for $u$ real and $\si>0$, 
may be unbounded from $\ghu{M}$ to $\lu{M}$ --- and, \emph{a fortiori}, on $\ghu{M}$.
Indeed, suppose for instance that $M$ is a complex symmetric space of the noncompact type.  
We argue by contradiction.  If $\cU_\si^{iu}$ were bounded from $\ghu{M}$ to $\lu{M}$
for some $u \neq 0$, then so would be the operator $\cL^{iu}$, because
$\cL^{iu} = \cU_{\si}^{iu} \, \big(\si+\cL\big)^{iu}$ and $\big(\si+\cL\big)^{iu}$ 
is bounded on $\ghu{M}$ \cite[Theorem~7]{MVo}. 
However, it is known \cite{MMV4} that $\cL^{iu}$ does not map
$\hu{M}$ to $\lu{M}$.
Since $\hu{M}$ is contained in $\ghu{M}$,
$\cL^{iu}$ does not map $\ghu{M}$ to $\lu{M}$ either. 
\end{remark} 

An important consequence of sectoriality of an operator $A$ on a Banach space is the boundedness of certain holomorphic functions of $A$.
More precisely, suppose that $0<\te\leq \pi$.  We denote by $H_0^\infty(S_\te)$ the space of all
bounded holomorphic functions on the sector $S_\te$ for which there exist positive constants~$C$ 
and $s$ such that 
$$
\bigmod{f(z)}
\leq C \, \frac{\mod{z}^s}{1+\mod{z}^{2s}}. 
\quant z \in S_\te;
$$
$H_0^\infty(S_\te)$ is called the \emph{Riesz--Dunford class} on $S_\te$.  The \emph{extended 
Riesz--Dunford class} $\cE(S_\te)$ is the Banach algebra generated by $H_0^\infty(S_\te)$,
the constant functions and the function $z\mapsto (1+z)^{-1}$.  For more on these classes
of functions, see \cite[pp.~27--29]{Haa}.  
Recall that if $A$ is a sectorial operator of angle $\om$ on a Banach space $\cY$ and $f$
belongs to the extended Riesz--Dunford class $\cE(S_\te)$ for some $\te > \om$, then $f(A)$
is bounded on $\cY$ \cite[Theorem~2.3.3]{Haa}. 

The functional calculus for sectorial operators is used in the proof of the
following proposition, which contains additional information on the operators $\cU_\si$.  

\begin{proposition} \label{p: invariance of the range} 
Assume that $ \si_1, \si_2, \ga >0$.
\begin{enumerate}
\item[\itemno1]
$\cU_{\si_1}^\gamma \big[\ghu{M}\big] = \cU_{\si_2}^\gamma \big[\ghu{M}\big]$.
\item[\itemno2]
There exists a constant $C$ such that
$$
C^{-1} \bignorm{\cU_{\si_1}^{-\gamma} f}{\mathfrak h^1} \leq \bignorm{\cU_{\si_2}^{-\gamma} f}{\mathfrak h^1} \leq C \bignorm{\cU_{\si_1}^{-\gamma} f}{\mathfrak h^1}  
$$
for every $f$ in $\cU^\gamma\big[\ghu{M}\big]$.
\end{enumerate}
\end{proposition} 
\begin{proof}
It is easily checked that the function $\varphi$ defined by
\[
\varphi(z) = \left(\frac{z+\si_1}{z+\si_2}\right)^\gamma
\]
belongs to the class $\cE(S_\te)$ for all $\theta \in (0,\pi)$. Since
$$
\cU_{\si_1}^{-\gamma} \cU_{\si_2}^\gamma = (\si_1+\cL)^\gamma (\si_2+\cL)^{-\gamma} = \varphi(\cL)
$$
on $L^2(M)$, and $\cL$ is sectorial of angle $\pi/2$ on $\ghu{M}$ by Theorem~\ref{t: bound semigroup O}~\rmii, we conclude by \cite[Theorem~2.3.3]{Haa} and Proposition \ref{p: uniflim} that $\cU_{\si_1}^{-\gamma} \cU_{\si_2}^\gamma$ extends to a bounded operator on $\ghu{M}$. Similarly one shows that $\cU_{\si_2}^{-\gamma} \cU_{\si_1}^\gamma$ extends to a bounded operator on $\ghu{M}$. Consequently the identities
\begin{gather}
\label{f: U_inversion}
\big(\cU_{\si_2}^{-\gamma}\cU_{\si_1}^\gamma\big)\,  \big(\cU^{-\gamma}_{\si_1}\cU_{\si_2}^\gamma\big)
= \cJ 
= \big(\cU_{\si_1}^{-\gamma}\cU_{\si_2}^\gamma\big) \, \big(\cU^{-\gamma}_{\si_2}\cU_{\si_1}^\gamma\big), \\
\label{f: U_exchange}
\cU_{\si_2}^\gamma = \cU_{\si_1}^\gamma \big( \cU_{\si_1}^{-\gamma} \cU_{\si_2}^\gamma \big), \qquad \cU_{\si_1}^\gamma = \cU_{\si_2}^\gamma \big( \cU_{\si_2}^{-\gamma} \cU_{\si_1}^\gamma \big),
\end{gather}
initially valid on $L^2(M)$, extend by density and boundedness to $\ghu{M}$. From \eqref{f: U_inversion} we deduce that the extensions of $\cU_{\si_2}^{-\gamma}\cU_{\si_1}^\gamma$ and $\cU_{\si_1}^{-\gamma}\cU_{\si_2}^\gamma$ are isomorphisms of $\ghu{M}$, and from this and \eqref{f: U_exchange} it follows that $\cU_{\si_1}\big[\ghu{M}\big] = \cU_{\si_2}\big[\ghu{M}\big]$. This proves \rmi.

From \eqref{f: U_exchange} we also deduce that, for all $f \in \cU^\gamma\big[\ghu{M}\big]$,
\[
\cU_{\si_1}^{-\gamma} f = \big( \cU_{\si_1}^{-\gamma} \cU_{\si_2}^\gamma \big) \cU_{\si_2}^{-\gamma} f, \qquad \cU_{\si_2}^{-\gamma}f = \big( \cU_{\si_2}^{-\gamma} \cU_{\si_1}^\gamma \big) \cU_{\si_1}^{-\gamma} f,
\]
and the $\ghu{M}$-boundedness of $\cU_{\si_1}^{-\gamma} \cU_{\si_2}^\gamma$ and $\cU_{\si_2}^{-\gamma} \cU_{\si_1}^\gamma$ gives \rmii.
\end{proof}

\section{A one-parameter family of Hardy type spaces} 
\label{s: A one parameter family of Hardy-type spaces}

\subsection{Definition and properties of \texorpdfstring{$\gXga{M}$}{Xgamma(M)}}
By Proposition~\ref{p: consequences of sectoriality}~\rmii,
the operator $\cU^\ga$ is bounded and injective on $\ghu{M}$ for all $\ga>0$.
Thus, the following definition makes sense.  

\begin{definition} \label{def: gXga}
Suppose that $\ga >0$.  
We denote by $\gXga{M}$ the space $\cU^\ga\big[\ghu{M}\big]$, endowed with 
the norm that makes $\cU^\ga$ an isometry, i.e., set 
$$
\bignorm{f}{\mathfrak X^{\gamma}}
:= \bignorm{\cU^{-\ga} f}{\mathfrak h^1}  
\quant f \in \cU^\ga\big[\ghu{M}\big].  
$$
\end{definition}

The following proposition gives some equivalent characterisations of the spaces $\gXga{M}$, showing in particular that replacing $\cU$ with $\cU_\si$ for some $\si>0$ in the above definition would determine the same spaces (up to equivalence of norms).

\begin{proposition} \label{p: useful property}
Let $\ga,\si > 0$. For a function $f$ on $M$, the following are equivalent:
\begin{enumerate}
\item[\itemno1] $f$ is in $\gXga{M}$;
\item[\itemno2] $f$ is in $\cU_\sigma^\gamma \big[\ghu{M}]$;
\item[\itemno3] both $f$ and $\cL^{-\ga} f$ are in $\ghu{M}$.
\end{enumerate}
Moreover there exists a positive constant $C$ independent of $f$ such that
\begin{gather}
\label{f: eq_U_Usi}
C^{-1} \bignorm{f}{\mathfrak X^{\gamma}} \leq \bignorm{\cU_\sigma^{-\gamma} f}{\mathfrak{h}^1} \leq C \bignorm{f}{\mathfrak X^{\gamma}}, \\
\label{f: eq_U_Lfrac}
C^{-1} \bignorm{f}{\mathfrak X^{\gamma}} \leq \bignorm{f}{\mathfrak{h}^1} + \bignorm{\cL^{-\gamma} f}{\mathfrak{h}^1} \leq C \bignorm{f}{\mathfrak X^{\gamma}}.
\end{gather}
\end{proposition} 
\begin{proof}
The equivalence of \rmi\ and \rmii\ and the inequalities \eqref{f: eq_U_Usi} are immediate consequences of Proposition \ref{p: invariance of the range}. It remains to prove the equivalence of \rmi\ and \rmiii, as well as the inequalities \eqref{f: eq_U_Lfrac}.

Assume first that both $f$ and $\cL^{-\ga} f$ belong to $\ghu{M}$.
Observe that, at least formally, 
\begin{equation}\label{f: Uga as vp Lga}
\cU^{-\ga} f
= \vp(\cL) \, \big(\cI+\cL^{-\ga}\big) f,
\end{equation}
where $\vp$ is given by
\[
\vp(z) = \frac{(1 + z)^\ga}{1+z^\ga}.
\]  
Now, $f$ and $\cL^{-\ga}f$ are in $\ghu{M}$ by assumption, whence so is
$\big(\cI+\cL^{-\ga}\big) f$.
On the other hand, it is straightforward to check that $\vp$ 
belongs to the class $\cE(S_\theta)$ for any $\theta \in (\pi/2,\pi)$. Since $\cL$ is a sectorial operator of angle $\pi/2$ on $\ghu{M}$ by Theorem \ref{t: bound semigroup O} \rmii, we deduce that $\vp(\opL)$ is bounded on $\ghu{M}$. Hence from \eqref{f: Uga as vp Lga} we conclude that $\cU^{-\ga} f$ is in $\ghu{M}$ and
\[
\bignorm{f}{\mathfrak X^{\gamma}} = \bignorm{\cU^{-\gamma} f}{\mathfrak h^1} \leq C \big( \bignorm{f}{\mathfrak{h}^1} + \bignorm{\cL^{-\gamma} f}{\mathfrak{h}^1} \big).
\]

Conversely, suppose that $\cU^{-\ga}f$ is in $\ghu{M}$.
Observe that $f = \cU^{\ga}\cU^{-\ga} f$.
Since $\cU^{-\ga}f$ is in $\ghu{M}$ by assumption, and $\cU^{\ga}$ is bounded on $\ghu{M}$
by Proposition~\ref{p: consequences of sectoriality}~\rmii, $f$ is in $\ghu{M}$, and
\[
\bignorm{f}{\mathfrak{h}^1} \leq C \bignorm{\cU^{-\gamma} f}{\mathfrak h^1} = \bignorm{f}{\mathfrak X^{\gamma}}.
\]
Furthermore, $\cL^{-\ga} f = (\cI+\cL)^{-\ga} \cU^{-\ga}f$, and $(\cI + \cL)^{-\ga}$
is bounded on $\ghu{M}$ by Theorem \ref{t: bound semigroup O} \rmiii.
Then $\cL^{-\ga} f$ is in $\ghu{M}$ and
\[
\bignorm{\cL^{-\ga} f}{\mathfrak{h}^1} \leq C \bignorm{\cU^{-\gamma} f}{\mathfrak h^1} = \bignorm{f}{\mathfrak X^{\gamma}},
\]
as required.
\end{proof}

The operators $\cU^\gamma$ commute with any other operator in the functional calculus of $\cL$, including resolvents and the heat semigroup. Hence, from Definition \ref{def: gXga} and Theorem \ref{t: bound semigroup O} one immediately obtains the following result. 

\begin{corollary} \label{t: bound semigroup O I}
The following hold:
\begin{enumerate}
\item[\itemno1]
the heat semigroup is uniformly bounded on $\gXga{M}$ for 
every $\ga>0$;
\item[\itemno2]
$\cL$ is a sectorial operator of angle $\pi/2$ on $\gXga{M}$;
\item[\itemno3]
$\ds \sup_{\la>0} \, \, \bigopnorm{\la \, (\la + \cL)^{-1}}{\mathfrak X^{\gamma}} < \infty$, equivalently
$\ds \sup_{\la>0} \bigopnorm{\cL \, (\la + \cL)^{-1}}{\mathfrak X^{\gamma}} < \infty$.
\end{enumerate}
\end{corollary}

A few other relevant properties of the spaces $\gXga{M}$ are stated below. One shoud compare parts \rmiii\ and \rmiv\ of Proposition \ref{p: Uga bounded on Hone} with the discussion in Remark \ref{rem: imaginarypowers}. In light of the interpolation property in part \rmv, the $\mathfrak{X}^\gamma$-$\mathfrak{h}^1$ boundedness of the imaginary powers of $\cL$ expressed in part \rmiv\ can be seen as an endpoint counterpart to their $L^p$-boundedness for $p \in (1,\infty)$.

\begin{proposition} \label{p: Uga bounded on Hone}
The following hold:
\begin{enumerate}
\item[\itemno1]
if $\Re z>0$, then $\cU^z$ is bounded 
 on $\gXga{M}$ for every $\ga>0$;
\item[\itemno2]
$\{\gXga{M}: \ga>0\}$ is a decreasing family of Banach spaces;
\item[\itemno3]
if $\ga >0$ and $u$ is real, then $\cU^{iu}$ is bounded from $\gXga{M}$ to $\ghu{M}$;
\item[\itemno4]
if $\ga >0$ and $u$ is real, then $\cL^{iu}$ is bounded from $\gXga{M}$ to $\ghu{M}$;
\item[\itemno5]
$\big(\fX^{\ga}(M),L^2(M)\big)_{[\te]} = L^{p_{\te}}(M)$, whenever $\te\in (0,1)$ and $p_{\te}=2/{(1-\te)}$.
\end{enumerate}
\end{proposition}

\begin{proof}
Observe that $\cU^z$ is bounded on $\gXga{M}$ if and only if 
$\big[\cU^\ga\big]^{-1} \cU^z \cU^\ga=\cU^z$ is bounded on $\ghu{M}$. Thus \rmi\ is an immediate consequence of Proposition \ref{p: consequences of sectoriality}~\rmii.

Next we prove \rmii.  By Proposition \ref{p: consequences of sectoriality}~\rmii\ and the definition of $\gXga{M}$, 
it is clear that $\ghu{M} \supseteq \gXga{M}$ for any $\ga > 0$, with continuous inclusion. So, if $\ga_2>\ga_1>0$, then 
\[
\fX^{\ga_2}(M)   
=        \cU^{\ga_2} \big[\ghu{M}\big]  
=        \cU^{\ga_2-\ga_1} \cU^{\ga_1}\big[\ghu{M}\big] 
=        \cU^{\ga_2-\ga_1} \big[ \fX^{\ga_1}(M) \big]
\subseteq  \fX^{\ga_1}(M),
\]
the last containment above being a consequence of the boundedness of 
$\cU^{\ga_2-\ga_1}$ on $\fX^{\ga_1}(M)$ proved in \rmi.

Notice that \rmiii\ is equivalent to the boundedness of $\cU^{\ga+iu}$ on $\ghu{M}$, so \rmiii\ is another consequence of \ref{p: consequences of sectoriality}~\rmii.

To prove \rmiv, notice that, by Proposition \ref{p: useful property}, $\cL^{iu}$ is bounded from $\gXga{M}$ to $\ghu{M}$
if and only if $\cL^{iu} \,\cU_{\si}^{\ga} $ is bounded on $\ghu{M}$ for some $\si>0$.
Note that, by spectral theory,
$$
\cL^{iu} \,\cU_{\si}^{\ga}
=  \cU_{\si}^{\ga+iu} \, \big(\si+\cL\big)^{iu};
$$
the operator $\cU_{\si}^{\ga+iu}$ is bounded on $\ghu{M}$ by \rmiii, 
and $\big(\si+\cL\big)^{iu}$
is bounded on $\ghu{M}$ because its symbol satisfies a Mihlin--H\"ormander
condition of any order on the strip $\{ z \in \BC : |\Im z| < \be\}$ provided $\si>\beta^2-b$ \cite[Theorem~7]{MVo},  
and \rmiv\ is proved.    

Finally, as already mentioned, $\cU$ is an isomorphism of $L^p(M)$ for all $p \in (1,2]$, hence so is $\cU^\ga$, while $\fX^{\ga}(M) = \cU^\ga\big[\ghu{M}\big]$, for all $\ga > 0$. The interpolation property \rmv\ for $\fX^{\ga}(M)$ is therefore an immediate consequence of the corresponding property for $\ghu{M}$ proved in \cite[Theorem 5]{MVo}.
\end{proof}

We shall show that $\{\gXga{M}: \ga>0\}$ is actually a strictly decreasing family of Banach spaces. To do so, we need to discuss an atomic decomposition of the spaces $\gXga{M}$.

\subsection{Atomic decomposition when \texorpdfstring{$\ga$}{gamma} is an integer}
 
In the case where $\ga$ is an integer, $\gXga{M}$ was defined in \cite{Vo}, where 
also some of its properties were investigated.  In particular, it was shown there that these spaces
admit an atomic decomposition that we now describe, which is a variant of the atomic decomposition 
for the spaces $\Xh{M}$ proved in \cite{MMV2}.
An atom $A$ in $\fX^k(M)$ will be a standard atom in $\ghu{M}$
satisfying an additional infinite dimensional cancellation condition,
expressed as orthogonality of $A$
to the space of $k$-harmonic functions in a neighbourhood of the support of $A$.

\begin{definition}
Suppose that $k$ is a positive integer and that
$B$ is a ball in $M$. We say that a function $V$ in $\ld{M}$
is \emph{$k$-harmonic} on $\OV{B}$ if
$\cL^k V$ is zero (in the sense of distributions)
in a neighbourhood of $\OV{B}$. We shall denote by $P_B^k$ the space of
$k$-harmonic functions
on $\OV{B}$.
Moreover, let $Q_B^k$ denote the space of $k$-quasi-harmonic functions 
on $\OV{B}$, i.e., the subspace of $L^2(M)$ consisting of all 
the functions $V$ such that $\cL^k V$ is constant (in the sense of distributions) 
in a neighbourhood of $\overline{B}$. 
\end{definition}

\begin{remark} \label{rem: elliptic reg}
By elliptic regularity, $P_B^k$ coincides with the space 
of the functions $V$ in $L^2(M)$ that are smooth in a 
neighbourhood of $\OV{B}$ and such that $\cL^k V$ is zero therein.
A similar remark applies to $Q_B^k$.
\end{remark}

A direct consequence of the definition of $P_B^k$ and $Q_B^k$ is the following chain
of inclusions:
$$
P_B^1 \subseteq Q_B^1\subseteq P_B^2 \subseteq Q_B^2 \subseteq \cdots\,;
$$ 
correspondingly
\[
(P_B^1)^\perp
\supseteq (Q_B^1)^\perp\supseteq (P_B^2)^\perp
\supseteq (Q_B^2)^\perp\supseteq  \cdots \,,
\]
where $(P_B^k)^\perp$ and $(Q_B^k)^\perp$ denote the orthogonal complements of $P_B^k$ and $Q_B^k$ is $L^2(M)$.

For each ball $B$ in $M$, let us denote by $L^2(B)$ the space of all $\ld{M}$ functions supported in $\overline{B}$. The following result is the counterpart for the spaces $P_B^k$ of \cite[Proposition 3.3]{MMV2}, where the case of $Q_B^k$ is treated; the proof is analogous and is omitted.

\begin{proposition} \label{p: canc II}
Suppose that $k$ is a positive integer, and that $B$ is a ball in $M$.
\begin{enumerate}
\item[\itemno1]
$
(P_B^k)^{\perp}
= \set{F\in L^2(M):\cL^{-k} F\in\ld{B}}.
$
\item[\itemno2]
$\cL^{-k} \bigl( (P_B^k)^{\perp})$ is contained in $\ld{B} \cap
\Dom(\cL^{k})$. Furthermore, functions in $(P_B^k)^{\perp}$
have support contained in $\OV{B}$.
\item[\itemno3]
$\cU^{-k} \bigl( (P_B^k)^{\perp})$ is contained in $\ld{B}$.
\end{enumerate}
\end{proposition}

\begin{definition} \label{def: Xk atoms}
Suppose that $k$ is a positive integer.
A \emph{standard $\fX^k$-atom} associated to
the ball $B$ of radius $\leq 1$ is a function $A$ in 
$\ld{M}$, supported in $B$, such that
\begin{enumerate}
\item[\itemno1]
$A$ is in $(Q_B^k)^{\perp}$;
\item[\itemno2]
$\ds\norm{A}{L^2}\leq \mu(B)^{-1/2}$.
\end{enumerate}
A \emph{global $\fX^k$-atom} associated to
the ball $B$ of radius $1$ is a function $A$ in $\ld{M}$, supported in $B$, 
such that 
\begin{enumerate}
\item[\itemno1]
$A$ is in $(P_B^k)^{\perp}$;
\item[\itemno2]
$\ds\norm{A}{L^2}\leq \mu(B)^{-1/2}$.
\end{enumerate}
An $\fX^k$-\emph{atom} is either a standard $\fX^k$-atom or a 
global $\fX^k$-\emph{atom}.
\end{definition}

\begin{remark}
An $\fX^k$-atom (standard or global) is also a 
standard $\mathfrak{h}^1$-atom: indeed, in either case the cancellation condition \rmi\ implies that the integral 
of $A$ vanishes, since $\chi_{2B}$ is in $P_B^k$ and in $Q_B^k$.
\end{remark}

\begin{remark}\label{rem: 12h atomi}
The set of $\fX^k$-atoms is a bounded subset of $\fX^k(M)$.
Indeed, in the case of a global $\fX^k$-atom $A$, from Proposition \ref{p: canc II} and \eqref{f: normah1funzioneL2} it follows immediately that
\[
\| \cU^{-k} A \|_{\mathfrak{h}^1} \leq C \bigopnorm{ \cU^{-k} }{L^2} \,;
\]
the fact that the same estimate also holds for a standard $\fX^k$-atom $A$ can be shown as in \cite[Remark 3.5]{MMV2}.
\end{remark}

\begin{definition}
Suppose that $k$ is a positive integer. The space $\Xxat{M}$ is the space of all functions $F$ in 
$\frh^1(M)$ that admit a decomposition of the form $F= \sum_j \la_j\, A_j$, where $\{\la_j\}$ is
a sequence in $\ell^1$ and $\{A_j\}$ is a sequence
of $\fX^k$-atoms. 
We endow
$\Xxat{M}$ with the norm
$$
\norm{F}{\Xnat}
= \inf\, \Bigl\{\sum_{j} \mod{\la_j}: F = \sum_{j} \la_j \, A_j,\quad
\hbox{$A_j$ $\fX^k$-atoms}\Bigr\}.
$$
\end{definition}

From Remark \ref{rem: 12h atomi} it is clear that $\Xxat{M} \subseteq \fX^k(M)$, with continuous embedding.
One can show that equality holds under a suitable geometric hypothesis on $M$.

\begin{definition}
We say that $M$ has $C^\ell$ bounded geometry if the injectivity radius is positive and the following hold:
\begin{enumerate}
\item[(a)] if $\ell=0$, then the Ricci tensor is bounded from below;
\item[(b)] if $\ell$ is positive, then the covariant derivatives 
$\nabla^j \hbox{Ric}$ of the Ricci tensor are uniformly bounded on $M$ for all $j \in \{0,\dots,\ell\}$.
\end{enumerate}
\end{definition}

The aforementioned atomic decomposition of $\gXh{M}$ is the content of the following theorem. We omit the proof, which follows the lines of the proof of \cite[Theorem~4.3]{MMV2}. 
\begin{theorem} \label{t: atomic dec}
Suppose that $k$ is a positive integer and that $M$ has $C^{2k-2}$ bounded geometry.
Then $\fX^k(M)$ and $\Xxat{M}$ agree as vector spaces and there exists a constant $C$ such that
\begin{equation} \label{f: atomic dec}
C^{-1} \, \norm{F}{\Xnat}
\leq \norm{F}{\fX^k}
\leq C \, \norm{F}{\Xnat}
\quant F \in \fX^k(M).
\end{equation}
\end{theorem}
Notice that when $k=1$ the geometric hypothesis of Theorem \ref{t: atomic dec} 
is already contained in our geometric assumptions on the manifold $M$ (see Section \ref{s: Background material}), 
so the atomic characterization of $\mathfrak X^1(M)$ holds without additional assumptions. 

As a consequence of the atomic characterization of the space $\mathfrak X^1(M)$ 
we can prove a result involving all $\fX^{\gamma}(M)$ spaces. 
Denote by $\cH^\infty(M)$ the space of all bounded harmonic functions on $M$, 
thought of as a subspace of $\gbmo{M}$, and consider the annihilator of $\cH^\infty(M)$ in $\ghu{M}$, defined~by 
$$
\cH^\infty(M)^\perp
:= \big\{ f \in \ghu{M}: \prodo{f}{H} = 0 \,\, \hbox{for all $H$ in $\cH^\infty(M)$} \big\},
$$
where $\prodo{\cdot}{\cdot}$ denotes the duality between $\ghu{M}$ and $\gbmo{M}$ (see \eqref{f: dualita}).

\begin{proposition}  
The following hold:
\begin{enumerate}
\item[\itemno1]
for every $\gamma>0$ the space $\gXga{M}$ is contained in $\cH^\infty(M)^\perp$\,;
\item[\itemno2]
$\{\gXga{M}: \ga>0\}$ is a strictly decreasing family of Banach spaces.
\end{enumerate}
\end{proposition}
\begin{proof}
We first show that $\gXu{M} \subseteq \cH^\infty(M)^\perp$.  By Theorem~\ref{t: atomic dec}, 
every function $f$ in $\gXu{M}$
admits a representation of the form  $\sum_j \, c_j\, A_j$, where the $A_j$'s are
$\mathfrak X^1$-atoms.  Each of these is annihilated by all bounded harmonic functions, so that, by \eqref{f: dualita},
$$
\begin{aligned}
\prodo{f}{H}
 = \sum_j \, c_j \,\int_M A_j \,H \wrt \mu 
= 0
\quant H \in \cH^\infty(M).
\end{aligned}
$$ 
Taking closures in $\ghu{M}$, we then obtain that $\OV{\gXu{M}}^{\,\ghu{M}} \subseteq \cH^\infty(M)^\perp$.  
Now, the sectoriality of $\cU$ implies \cite[Proposition~3.1.1~(d)]{Haa} that 
$\OV{\gXu{M}}^{\,\ghu{M}} = \OV{\gXga{M}}^{\,\ghu{M}}$, and we can conclude that 
$\OV{\gXga{M}}^{\,\ghu{M}} \subseteq \cH^\infty(M)^\perp$. This proves \itemno1.

We now prove \rmii. Let $\ga_2>\ga_1>0$. 
In view of Proposition \ref{p: Uga bounded on Hone} it remains to show that the containment 
$\fX^{\ga_2}(M) \subseteq \fX^{\ga_1}(M)$ is proper. 
We argue by contradiction.  Suppose that $\fX^{\ga_2}(M) = \fX^{\ga_1}(M)$.  Since 
$\cU$ is an injective sectorial operator on $\ghu{M}$ 
(see Proposition~\ref{p: consequences of sectoriality}),  
$\big(\cU^{\ga_1}\big)^{-1} = \big(\cU^{-1}\big)^{\ga_1}= \cU^{-\ga_1}$ 
\cite[Propositions~3.1.1~(e) and 3.2.1~(a)]{Haa}. 
Furthermore, $\cU^{-\ga_1}\cU^{\ga_2} \subseteq \cU^{\ga_2-\ga_1}$ \cite[Proposition~3.2.1~(b)]{Haa}. 
Since $\fX^{\ga_2}(M) = \fX^{\ga_1}(M)$, the operator $\cU^{-\ga_1}\cU^{\ga_2}$ is surjective
on $\ghu{M}$, so $\cU^{\ga_2-\ga_1}$ is also surjective on $\ghu{M}$.  
By \cite[Proposition 3.1.1 (d)]{Haa}, $\cU^{\ga_2-\ga_1}$ is bounded and injective on $\ghu{M}$, 
whence $\cU^{\ga_2-\ga_1}$ is an isomorphism of $\ghu{M}$.  
Therefore $0$ is in the resolvent set of $\cU^{\ga_2-\ga_1}$.
However, $0$ is in the $\mathfrak h^1$-spectrum of $\cU$ ($\cU$ cannot be surjective, 
because, by \rmi, $\cU \mathfrak h^1(M)=\mathfrak X^1(M)$ is 
contained in the annihilator of constant functions), and this contradicts 
the spectral mapping theorem \cite[Proposition~3.1.1~(j)]{Haa}.
\end{proof}

\subsection{Lack of atomic decomposition when \texorpdfstring{$\ga$}{gamma} is not an integer}

In this subsection we restrict our analysis to symmetric spaces of the noncompact type $\BX$ of real rank one, and
show that, if $\ga$ is not a positive integer, then $\gXga{\BX}$ does not admit an
atomic decomposition. A similar result holds for the analogue of $\gXga{\BX}$
on homogeneous trees (see \cite[Theorem~5.8]{CM} for details). 
For the notation and the main properties of noncompact symmetric spaces and 
for spherical analysis thereon we refer the reader to \cite{H1, H}. 

We recall here that $\BX$ is a quotient $G/K$, where $G$ is a noncompact semisimple Lie group 
of finite centre and real rank one and $K$ is a maximal compact subgroup of $ G$.  
Given a Cartan decomposition $\frg=\frp\oplus \frk$ of the 
Lie algebra of $G$, we denote by $\fra$ a maximal abelian subspace of 
$\frp$ and by $\fra^*_{\BC}$ the complexification of its dual.
The real rank one assumption means that $\dim \fra = 1$ and implies that the Weyl group is $\{1,-1\}$.
If $M$ denotes the centralizer of $\fra$ in $K$, we denote by $B$ the quotient $K/M$. 
For every compactly supported function $f$ on $\BX$, its 
Helgason--Fourier transform $\wt f$ is a function on $\fra^*_{\BC}\times B$ 
defined as in \cite[p.\ 223]{H}. The Paley--Wiener theorem on noncompact symmetric spaces will be a 
key ingredient in the proof of the following result.

\begin{theorem}\label{t: no atomic}
Suppose that $\BX$ is a symmetric space of the noncompact type and real rank one, 
that $k$ is a positive integer, and that $\ga$ is in $(k-1,k)$.  
\begin{itemize}
\item[\itemno1] If $f$ is a function in $\gXga{\BX}$ with compact support, then $f$ belongs to $\gXh{\BX}$.
\item[\itemno2] If $f$ is a function in $\ld{\BX}\cap \gXga{\BX}$ with compact support contained in the ball ${B_1(o)}$, then $f$ is a multiple of an $\fX^{k}$-atom.
\item[\itemno3] $\gXh{\BX}$ is not dense in $\gXga{\BX}$.
\end{itemize}
\end{theorem}

\begin{proof}
We prove the result in the case where $k=1$.  The proof for $k\geq 2$ is similar
and is omitted.  

We first prove \rmi. We show that if $\ga$ is in $(0,1)$, and $f$ is a function in $\gXga{\BX}$
with compact support, then $f$ belongs to $\gXu{\BX}$.
Suppose that the support of $f$ is contained in the ball $\OV{B_R(o)}$ for some $R>0$.
Then, by the Paley--Wiener theorem for the Helgason--Fourier transform \cite[Corollary~5.9, p.~281]{H} and the 
fact that $f$ is in $\lu{\BX}$, $\wt{f}(\cdot,b)$ extends
to an entire function of exponential type $R$ uniformly in $b$ and there exist constants $C$ and $N$ such that 
\begin{equation} \label{f: est PW}
\bigmod{\wt f(\la,b)}
\leq C \, (1+|\lambda|)^N \e^{\mod{\Im \la}R} 
\quant \la \in \astar_\BC \quant b \in B.
\end{equation} 
Furthermore, $\wt f$ is smooth on $\astar_\BC\times B$ and satisfies the following symmetry 
condition:
\begin{equation} \label{f: symmetry condition}
\int_B \e^{(-i\la+\rho)(A(x,b))} \,  \wt f(-\la, b) \wrt b
= \int_B \e^{(i\la+\rho)(A(x,b))} \, \wt f(\la, b) \wrt b
\end{equation}
for every $x$ in $\BX$ and every $\la$ in $\astar_\BC$; here $\rho \in \astar$ denotes as usual half the sum of the positive restricted roots, while $A : \BX \times B \to \fra$ is defined as in \cite[p.\ 223]{H}.

Since $f$ is in $\gXga{M}$, by Proposition \ref{p: useful property} there exists a function $g$ in $\ghu{\BX}$
such that
\[
f = \cL^\ga g.
\]  
Since $g$ is in $\lu{\BX}$, its Helgason--Fourier transform
$\wt{g}(\cdot,b)$ is a continuous function on $\astar + i[-1,1]\rho$ for almost all $b$ in $B$ \cite{H3,SS}, and
\begin{equation}\label{f: HelgasonFourier transform}
\wt f(\la, b)
= Q(\la)^\ga \,\, \wt g(\la,b) \,,
\end{equation}
where $Q$ is the quadratic form on $\astarc\times \astarc$ defined by
$$
Q(\la)
= \prodo{\la}{\la} + \prodo{\rho}{\rho}
$$
(see \cite[Lemma 1.4, p.\ 225]{H} and \cite[Section 1]{CGM}).
Note that $Q$ vanishes at the points of 
$\pm i \rho$;
hence from \eqref{f: HelgasonFourier transform} we deduce that $\wt{f}(\pm i\rho,b)=0$ for almost all $b$ in $B$, and actually this holds for all $b$ in $B$ because $\wt{f}$ is smooth on $\astar_\BC\times B$.
Since $\wt{f}(\cdot,b)$ is entire, its zeros must have at least order 1.
Moreover $1/Q$ is a meromorphic function in $\astar_\BC$ with simple poles at $\pm i \rho$.  
Therefore $\la \mapsto Q(\la)^{-1} \,\, \wt{f}(\la,b)$ is an entire function for every $b$ in $B$. 
Since $Q(-\la) = Q(\la)$ for all $\la$ in $\astar_\BC$, 
from the symmetry condition \eqref{f: symmetry condition} we deduce that
$$
\int_B \e^{(-i\la+\rho) A(x,b)} \,  Q(-\la)^{-1} \wt f(-\la, b)\wrt b
= \int_B \e^{(i\la+\rho) A(x,b)} \, Q(\la)^{-1} \, \wt f(\la, b)\wrt b.  
$$
Furthermore, $(\la,b) \mapsto Q(\la)^{-1} \,\, \wt{f}(\la,b)$ 
is clearly smooth on $\astar_\BC \times B$, and satisfies the estimate \eqref{f: est PW} (possibly
with a different constant $C$).  
Again by the Paley--Wiener theorem for the Helgason--Fourier transform,
 there exists a distribution $h$ on $\mathbb X$ supported in $\OV{B_R(o)}$ such that
$$
\ds\wt{h}(\la,b) = Q(\la)^{-1} \,\,\wt{f}(\la,b) \,,
$$   
that is, $h=\mathcal L^{-1}f$. By assumption $f\in\mathfrak X^{\ga}(\mathbb X)\subseteq \mathfrak h^1(\mathbb X)
\subseteq L^1(\mathbb X)$. 
Hence, by \cite[Theorem 4.7]{CGM}, $\cL^{-1}f$ 
is a function in $L^p(\mathbb X)$ for every $p\in \big(1,n/(n-2)\big)$.
Since $\opL^{-1} f$ is supported in $\OV{B_R(o)}$, by \eqref{f: normah1funzioneL2}
we deduce that $\cL^{-1}f \in \mathfrak h^1(\mathbb X)$.
By Proposition \ref{p: useful property} we conclude that $f$ 
belongs to $\gXu{\BX}$. 

\smallskip
We now prove \rmii. Suppose that $f \in \ld{\BX}\cap \gXga{\BX}$ and has support contained in the ball $B = B_1(o)$. Then $f$ is integrable, so from the proof of \rmi\ it follows that $\cL^{-1} f$ is supported in $\OV{B}$, and moreover $\cL^{-1} f$ is in $L^2(M)$ because $\cL^{-1}$ is $L^2$-bounded. From Proposition \ref{p: canc II} \rmi\ we deduce that $f$ is in $(P_B^1)^\perp$, and consequently $f$ is a multiple of a global $\mathfrak{X}^1$-atom.

\smallskip
Finally we prove \rmiii, i.e., we show that $\gXu{\BX}$ is not dense in $\gXga{\BX}$. 
Consider $\OV{\gXu{\BX}}^{\, \gXga{\BX}}$, i.e., the closure of $\gXu{\BX}$ in $\gXga{\BX}$.    
Since $\cU^{-\ga}$ is an isometry between $\gXga{\BX}$ and 
$\ghu{\BX}$, the topology of $\gXga{\BX}$ is transported by $\cU^{-\ga}$ to that of $\ghu{\BX}$, and 
$\cU^{-\ga}\, \big(\OV{\gXu{\BX}}^{\, \gXga{\BX}} \big) = \OV{\cU^{-\ga}\gXu{\BX}}^{\, \ghu{\BX}}$.  
Clearly $\cU^{-\ga} \gXu{\BX} = \fX^{1-\ga}(\BX)$ 
and we obtain that 
$$
\cU^{-\ga}\, \big(\OV{\gXu{\BX}}^{\, \gXga{\BX}} \big) =\OV{\fX^{1-\ga}(\BX)}^{\, \ghu{\BX}}.
$$   
However, Proposition~\ref{p: Uga bounded on Hone}~\rmvi\ ensures that $\fX^{1-\ga}(\BX)$ 
is properly contained in $\cH^\infty(\BX)^\perp$, 
so that $\OV{\fX^{1-\ga}(\BX)}^{\, \ghu{\BX}} \subseteq \cH^\infty(\BX)^\perp$,   
which we already know to be a proper closed subspace of $\ghu{\BX}$.  Thus, 
$\OV{\gXu{\BX}}^{\, \gXga{\BX}} \subseteq \cU^{\ga}\cH^\infty(\BX)^\perp \subsetneq \cU^\ga \ghu{\BX} = \gXga{\BX}$.    
\end{proof}

\section{Modified Riesz--Hardy space}
\label{s: Modified Riesz--Hardy space}

In this section, we consider the modified Riesz--Hardy space $\wthuR{M}$, defined by
$$
\wthuR{M}
:=  \big\{f \in \ghu{M}: \mod{\cR f} \in \lu{M} \big\},
$$
where $\cR$ denotes the ``geometric" Riesz transform $\nabla \cL^{-1/2}$.
Clearly $\wthuR{M}\subseteq \huR{M}$, where $\huR{M}$ is the Riesz--Hardy space defined in \eqref{f: rieszhardy}.
The main result of this section is the characterisation of $\wthuR{M}$ as the space $\mathfrak{X}^{1/2}(M)$ introduced in the previous section.

An important ingredient in the proof of this characterisation of  $\wthuR{M}$ is the following boundedness result.

\begin{proposition} \label{p: resolvent bounded}
The operator $(\si+\cL)^{-1/2}$ is bounded from $\lu{M}$ to $\ghu{M}$ for all $\si>\be^2-b$.   
\end{proposition}

In the proof of this proposition 
we shall apply the following lemma, which is a slight variant of \cite[Lemma 2.4]{MMV0} and \cite[Lemma 5.1]{MMV1}. 
Let 
$\cJ_{\nu}(t)=t^{-\nu} J_{\nu}(t)$\,, where $J_{\nu}$ denotes the standard Bessel function 
of the first kind and order $\nu$, and let $\cO$ denote the differential operator $t \partial_t$ on the real line.

\begin{lemma}\label{l: intbyparts}
For every positive integer $N$ there exists a polynomial $P_N$ of degree $N$ without constant term such that
$$
\int_{-\infty}^{+\infty}f(t)\cos(tv)\wrt t=\int_{-\infty}^{+\infty}P_N(\cO)f(t) \, \cJ_{N-1/2}(tv)\wrt t\,,
$$
for all compactly supported functions $f$ such that $\cO^{\ell}f\in L^1(\mathbb R)$ for every $\ell=0,\dots,N$. 
\end{lemma} 

\begin{proof}[Proof of Proposition \ref{p: resolvent bounded}]
Suppose that $g$ is in $\lu{M}$. Observe that we can write $g=\sum_jg_j$, 
where each $g_j$ is supported in a ball of radius $1$ and $\|g\|_{L^1}=\sum_j\|g_j\|_{L^1}$\,.
It suffices to prove that $(\si+\cL)^{-1/2}g_j$ is in 
$\ghu{M}$, and that there exists a constant $C$, independent of $g$ and $j$, such that 
$$
\bignorm{(\si +\cL)^{-1/2}g_j}{\mathfrak h^1}
\leq C \, \bignorm{g_j}{L^1}\,. 
$$

For the sake of convenience, in the rest of this proof we suppress the index $j$ and 
write $h$ instead of $g_j$.  We assume that the support of $h$ is contained in the ball
$B_1(o)$.

We argue as in the proof of \cite[Lemma~4.2]{MMV2}.  
Denote by $\om$ an even function in 
$C_c^\infty(\BR)$ which is supported in $[-3/4,3/4]$, is equal to~1
in $[-1/4,1/4]$, and satisfies 
$$
\sum_{j\in \BZ} \om(t-j) = 1
\quant t \in \BR.
$$
Define $\om_0 := \om$, and, for each $j$ in $\{1,2,3,\ldots\}$, 
\begin{equation}\label{omj}
\om_j(t) 
:= \om (t-j) + \om(t+j) \quant t \in \BR.
\end{equation}
Observe that the support of $\om_j$ is contained in the set of all $t$ in $\BR$ such that 
$j-3/4\le\mod{t}\le j+3/4$.  Set $m(\la) = (c^2+\la^2)^{-1/2}$, where $c:=\sqrt{\si+b}>\be$.  Recall that 
\begin{equation}\label{f: thirdBessel}
\wh m(t) 
= K_0(c|t|)
\quant t \in \BR,
\end{equation}
where $K_\nu$ denotes the modified Bessel function of the third kind and order $\nu$ (see, for instance,
\cite[p.~108]{Le}).

As in the proof of Theorem \ref{t: bound semigroup O}, let $\cD = \sqrt{\cL - b}$.
By the spectral theorem, we write $(\si+\cL)^{-1/2} = m(\cD)=\sum_{j=0}^{\infty}T_j(\cD)$\,, where
\begin{equation} \label{f: Bj}
T_j(\la) 
= \ir  (\om_j\, \wh m)(t)  \, \cos(  t \la) \wrt t \quant \la \in \BR.
\end{equation}
An induction argument, based on formulae \cite[(5.7.9), p.~110]{Le} and \cite[(5.7.12), p.~111]{Le}, 
and the estimates \cite[formulae (5.7.12) and (5.11.9)]{Le} 
show that, for all $\ell\in\mathbb N$, the function $\cO^{\ell}K_0$ is in $L^1(\mathbb R^+)$, and, 
for all $\ell\geq 1$, $\cO^{\ell}K_0$ is bounded;
moreover, for all $\varepsilon \in (0,1)$ and $\ell\in\mathbb N$, there exists a positive constant $C$ such that 
\begin{equation}\label{f: K0infty}
|\cO^{\ell}K_0(t)|\leq C \,\e^{-\varepsilon t}\qquad \forall t\geq 1/4\,.
\end{equation}
In view of \eqref{f: thirdBessel}, it is straightforward to check that $\cO^\ell(\om_j\, \wh m)$ 
is in $\lu{\BR}$ for all nonnegative integers $\ell$ and $j$, and that $\cO^\ell(\om_j\, \wh m)$ is bounded whenever $\ell+j\geq 1$. 
By Lemma \ref{l: intbyparts}, for every positive integer $N$, for all $j\geq 0$ and $\la\in\mathbb R$
\begin{equation}\label{Tjla}
T_j(\la)=\int_{-\infty}^{+\infty}P_N(\cO)(\om_j\, \wh m)(t)\cJ_{N-1/2}(t\la)\wrt t\,,
\end{equation}
where $P_N$ is a polynomial of degree $N$ without constant term.  
 
By \eqref{f: K0infty}, for all $N$ there exist positive constants $C$ and $c'\in (\beta,c)$ 
such that, for $j = 1,2,3,\dots$,
\begin{equation}  \label{f: decay wh m}
\bigmod{P_N(\cO)(\om_j\, \wh m)(t)}
\leq C \,  \e^{-c'\mod{t}} \,  
\end{equation} 
on the support of $\om_j$. By the asymptotics of $\cJ_{N-1/2}$ \cite[formula (5.11.6)]{Le},
$$
\sup_{s>0} \mod{(1+s)^{N} \, \cJ_{N-1/2}  (s)} < \infty.
$$
Let $k_{\cJ_{N-1/2} (t\cD)}$ denote the Schwartz kernel of the operator $\cJ_{N-1/2} (t\cD)$.
If we choose $N>(n+2)/2$, we may apply \cite[Proposition~2.2~\rmi]{MMV0} and conclude that   
\begin{equation}\label{f: Jh2}
\begin{aligned}
\bignorm{\cJ_{N-1/2} (t\cD) h}{L^2}
&\leq \bignorm{h}{L^1}\,\sup_{y\in M}  \bignorm{k_{\cJ_{N-1/2} (t\cD)}(\cdot,y)}{L^2} \\
&\leq C\, \bignorm{h}{L^1}\,\mod{t}^{-n/2}\, \bigl(1+\mod{t}\bigr)^{n/2} 
\end{aligned}
\end{equation}
for every $t \in \BR\setminus \{0\}$.  Take $j\geq 1$ and observe that the support of $P_N(\cO)(\om_j\, \wh m)$ 
is contained in $\set{t\in\BR:j-3/4\le\mod{t}\le j+3/4}$. 
Hence 
\begin{align}
\bignorm{T_j(\cD)h}{L^2}
& \leq  \ir  \bigmod{P_N(\cO)(\om_j\, \wh m)(t)}  
     \,  \bignorm{\cJ_{N-1/2}  (t \cD)h}{L^2}  \wrt t \nonumber \\
& \leq C \bignorm{h}{L^1}\, \int_{j-3/4}^{j+3/4} \bigmod{{P_N(\cO)(\om_j\, \wh m)}(t)}  \,
         \bigmod{t}^{-n/2}\, \bigl(1+\mod{t}\bigr)^{n/2} \wrt t \\
& \leq  C \, \e^{-c' j} \bignorm{h}{L^1}
    \quant j \in \{1,2,\ldots\}\nonumber.
\end{align}
In the last inequality we have used \eqref{f: Jh2} and  \eqref{f: decay wh m}. 
By \eqref{Tjla} and finite propagation speed, $T_j(\cD)h$ is a function in $\ld{M}$ 
with support contained in $B_{j+1}(o)$; so, by \eqref{f: normah1funzioneL2},
$$
\bignorm{T_j(\cD)h}{\mathfrak h^1}\leq C\,\bignorm{h}{L^1}\, \e^{-c' j}\,j^{\al/2}\,\e^{\be j}\,,
$$
for every $j \in \{1,2,\ldots\}$.  Hence 
$$
\begin{aligned}
\Bignorm{\sum_{j=1}^\infty T_j(\cD)  h}{\mathfrak h^1}  
\leq  C \bignorm{h}{L^1}\, \sum_{j=1}^\infty  j^{\al/2} \, \e^{(\be-c')j} \leq C \bignorm{h}{L^1} \,, 
\end{aligned}
$$
because $c'>\be$.  

It remains to estimate $T_0(\cD)$. Observe that, for all $t\in\BR\setminus \{0\}$, 
since $\cJ_{N-1/2}(t\,\cdot)$ is entire of exponential type $|t|$, 
by finite propagation speed $k_{\cJ_{N-1/2}(t\cD)}(\cdot,y)$ is supported in $\overline{B_{|t|}(y)}$. 
Hence, by H\"older's inequality and \cite[Proposition~2.2~\rmi]{MMV0}, for every $t\in [-1,1] \setminus \{0\}$ and $p \in [1,2]$,
$$
\begin{aligned}
\bignorm{\cJ_{N-1/2}  (t \cD)h}{L^p}
&\leq \bignorm{h}{L^1} \sup_{y\in M}  \bignorm{k_{\cJ_{N-1/2} (t\cD)}(\cdot,y)}{L^p}\\
&\leq C\,\bignorm{h}{L^1} |t|^{n(1/p-1/2)}\,\sup_{y\in M}  \bignorm{k_{\cJ_{N-1/2} (t\cD)}(\cdot,y)}{L^2}\\
&\leq C\,\bignorm{h}{L^1} |t|^{-n/p'}  \,.
 \end{aligned}
$$
Therefore, we see that, if $p'>n$, then
$$
\begin{aligned}
\bignorm{T_0(\cD)h}{L^p}
& \leq  \ir  \bigmod{P_N(\cO)(\om_0\, \wh m)(t)}  \,  \bignorm{\cJ_{N-1/2}  (t \cD)h}{L^p}  \wrt t \\
& \leq C \bignorm{h}{L^1} \,\int_{-1}^{1} \bigmod{{P_N(\cO)(\om_0\, \wh m)(t)}}  \, \bigmod{t}^{-n/p'}\,\wrt t \\
& \leq C \bignorm{h}{L^1}\,;
\end{aligned}
$$
observe that $P_N(\cO)(\om_0\, \wh m)$ is bounded since $P_N$ has no constant term. Again by finite propagation speed, $T_0(\cD)h$ is supported in ${B_2(o)}$, so from \eqref{f: normah1funzioneL2} we deduce that
$$
\bignorm{ T_0(\cD)  h}{\mathfrak h^1}  \leq C \bignorm{h}{L^1}\,,
$$ 
 as required to conclude the proof. 
\end{proof}

\begin{theorem}  \label{t: modified HR}
The modified Riesz--Hardy space $\wthuR{M}$ agrees with $\gXum{M}$.
\end{theorem}

\begin{proof}
First we prove that $\gXum{M} \subseteq \wthuR{M}$.  Suppose that $f$ is in $\gXum{M}$, and let $\si>0$.
Then, by Proposition \ref{p: useful property}, there exists $g$ in $\ghu{M}$ such that $\cU_\si^{1/2} g = f$, and therefore
$$
\begin{aligned}
\bigmod{\cR f}
& = \bigmod{\nabla (\si+ \cL)^{-1/2} g}.
\end{aligned}
$$
It is well known that the local Riesz transform $\nabla (\si+\cL)^{-1/2}$
maps $\ghu{M}$ to $\lu{M}$, provided $\si$ is large enough \cite[Theorem 8]{MVo}, hence $\bigmod{\cR f}$ belongs to $\lu{M}$.  Furthermore,
$f$ itself belongs to $\ghu{M}$, because $\gXum{M} \subseteq \ghu{M}$.
Thus, $f$ is in $\wthuR{M}$, as required.  

Next we prove that $\wthuR{M} \subseteq \gXum{M}$.  Suppose that $f$ is in $\wthuR{M}$.
Then $f$ is in $\ghu{M}$ and $\bigmod{\cR f}$ is in $\lu{M}$.  By the inequality \eqref{f: FF}, $\cL^{-1/2} f$ is in $\lu{M}$.  
Choose $\si > \be^2-b$.
Notice that 
$\ds \frac{\sqrt{\si+z}}{\sqrt z} = \vp(z) + \frac{\sqrt{\si}}{\sqrt z}$, where
the function $\vp$ belongs to the class $\cE(S_\te)$ for all $\te$ in $(0,\pi)$.
Since $\cL$ generates a contraction semigroup on $L^1(M)$, $\cL$ is a sectorial operator of angle $\pi/2$ on $\lu{M}$, and therefore $\vp(\cL)$ is bounded on $\lu{M}$ 
\cite[Theorem~2.3.3]{Haa}.  We already know that $\cL^{-1/2}f$ belongs to $\lu{M}$, whence $\cU_{\si}^{-1/2}f = \vp(\cL) f + \sqrt{\si} \cL^{-1/2} f$
is in $\lu{M}$.

Now, set $g = \cU_{\si}^{-1/2}f$.  Then  
$
\cL^{-1/2} f 
= (\si+\cL)^{-1/2} g
$, which 
belongs to $\ghu{M}$ by Proposition~\ref{p: resolvent bounded}.
Since $f$ belongs to $\ghu{M}$ by assumption, we conclude that $f \in \ghu{M}$ by Proposition \ref{p: useful property}.
\end{proof}


\begin{thebibliography}{HLMMY}

\bibitem[A]{A1} J.-Ph. Anker,
Sharp estimates for some functions of the Laplacian
on noncompact symmetric spaces,
\emph{Duke Math. J.} \textbf{65} (1992), 257--297.

\bibitem[AMR]{AMR} P. Auscher, A. McIntosh and E. Russ,
Hardy spaces of differential forms on Riemannian manifolds,
\emph{J. Geom. Anal.} \textbf{18} (2008), 192--248.

\bibitem[B]{B} H. Brezis, 
\emph{Functional analysis, Sobolev spaces and partial differential equations},
Universitext, Springer, New York, 2011.

\bibitem[Br]{Br} R. Brooks,
A relation between growth and the spectrum of the Laplacian,
\emph{Math. Z.} \textbf{178} (1981), 501--508.

\bibitem[BDL]{BDL} T.A. Bui, X.T. Duong and F.K. Ly,
Maximal function characterizations for new local Hardy-type spaces on spaces of homogeneous type,
\emph{Trans. Amer. Math. Soc} \textbf{370} (2018), 7229--7292.

\bibitem[BGS]{BGS} D.L. Burkholder, R.F. Gundy and M.L. Silverstein,
A maximal function characterization of the class $H^p$,
\emph{Trans. Amer. Math. Soc.} \textbf{157} (1971), 137--153.

\bibitem[CMM1]{CMM1} A. Carbonaro, G. Mauceri and S. Meda, 
$H^1$ and $BMO$ for certain locally doubling metric measure spaces,  
\emph{Ann. Scuola Norm. Sup. Pisa Cl. Sci.} \textbf{8} (2009), 543--582.

\bibitem[CMM2]{CMM2} A. Carbonaro, G. Mauceri and S. Meda, 
$H^1$ and $BMO$ for certain locally doubling metric measure
spaces of finite measure, \emph{Colloq. Math.}
\textbf{118} (2010), 13--41.

\bibitem[Ce]{Ce} D. Celotto, 
\emph{Riesz transforms, spectral multipliers and Hardy spaces on graphs}, 
Ph.D.\ Thesis, Universit\`a di Milano-Bicocca, 2016, \texttt{https://boa.unimib.it/handle/10281/118889}. 

\bibitem[CM]{CM} D. Celotto and S. Meda, 
On the analogue of the Fefferman--Stein theorem on graphs with the Cheeger property, 
\emph{Ann. Mat. Pura Appl.} \textbf{197} (2018), 1637--1677.

\bibitem[Ch]{Ch} I. Chavel, \emph{Isoperimetric inequalities}, Cambridge Tracts in Mathematics \textbf{145}, Cambridge University Press, 2001.

\bibitem[CG]{CG} M.~Christ and D.~Geller, 
Singular integral characterizations of Hardy spaces on homogeneous groups,
\emph{Duke Math. J.} \textbf{51} (1984), 547--598.

\bibitem[Coi]{Coi} R.R.~Coifman, A real variable characterisation of $H^p$,
\emph{Studia Math.} \textbf{51} (1974), 269--274.

\bibitem[CGM]{CGM} M. Cowling, S. Giulini and S. Meda, $L^p$-$L^q$ estimates for functions of the Laplace-Beltrami
operator on noncompact symmetric spaces. I, \emph{Duke Math. J.} \textbf{72} (1983), 109--150. 

\bibitem[CW]{CW} R.R. Coifman and G. Weiss, Extension of Hardy spaces and their use in analysis,
\emph{Bull. Amer. Math. Soc.} \textbf{83} (1977), 569--645.

\bibitem[DKKP]{DKKP} S. Dekel, G. Kerkyacharian, G. Kyriazis and P. Petrushev, Hardy spaces
associated with non-negative self-adjoint operators, \emph{Studia Math.} \textbf{239} (2017), 17--54.

\bibitem[DJ]{DJ} J.~Dziuba\'nski and K.~Jotsaroop, On Hardy and BMO spaces for Grushin operator,  
\emph{J. Fourier Anal. Appl.} \textbf{22} (2016), 954--995.

\bibitem[DW]{DW} J.~Dziuba\'nski and B.~Wr\'obel,  Strong continuity on Hardy spaces, 
\emph{J. Approx. Th.} \textbf{211} (2017), 85--93.

\bibitem[DZ1]{DZ1} J.~Dziuba\'nski and J.~Zienkiewicz,  Hardy space $H^1$ associated to Schr\"odinger operator with
potential satisfying reverse H\"older inequality,
\emph{Rev. Mat. Iberoam.} \textbf{15} (1999), 279--296.

\bibitem[DZ2]{DZ2} J.~Dziuba\'nski and J.~Zienkiewicz,  A characterization of Hardy spaces associated with certain
Schr\"odinger operators,
\emph{Potential Anal.} \textbf{41} (2014), 917--930.

\bibitem[FS]{FS} C. Fefferman and E.M.~Stein, Hardy spaces of several variables, 
\emph{Acta Math.} \textbf{129} (1972), 137--193.

\bibitem[Go]{Go} D. Goldberg, A local version of real Hardy spaces, 
\emph{Duke Math. J.} \textbf{46} (1979), 27--42.

\bibitem[Gr]{Gr1} A. Grigor'yan,  Estimates of heat kernels on 
Riemannian manifolds, in \emph{Spectral Theory and Geometry},
ICMS Instructional Conference Edinburgh 1988
(B.~Davies and Y.~Safarov eds.), London Mathematical Society Lecture
Note Series \textbf{273}, Cambridge University Press, 1999.

\bibitem[Haa]{Haa} M.~Haase, 
\emph{The functional calculus for sectorial operators},
Operator theory, Advances and applications \textbf{169}, Birkh\"auser Verlag, 2006.  

\bibitem[HMY]{HMY} Y. Han, D. M\"uller and D. Yang, A theory of Besov and Triebel--Lizorkin spaces on metric measure spaces
modeled on Carnot--Carathéodory spaces. \emph{Abstr. Appl. Anal.} (2008), Article ID 893409.

\bibitem[H1]{H1}  S.~Helgason,
\emph{Groups and geometric analysis}, Academic Press, New York, 1984.

\bibitem[H2]{H} S.~Helgason,
\emph{Geometric analysis on symmetric spaces},
Math. Surveys \& Monographs \textbf{39}, American Mathematical Society, 1994.

\bibitem[H3]{H3} S.~Helgason,
The Abel, Fourier and Radon transforms on symmetric spaces,
\emph{Indag. Math. (N.S.)} \textbf{16} (2005), 531--551.

\bibitem[HLMMY]{HLMMY} 
S. Hofmann, G. Lu, D. Mitrea, M. Mitrea and L. Yan, Hardy spaces associated to
non-negative self-adjoint operators satisfying Davies--Gaffney estimates, 
\emph{Mem. Amer. Math. Soc.} \textbf{214} (2011), no. 1007.

\bibitem[Io]{Io}  A.D. Ionescu, Fourier integral operators on noncompact 
symmetric spaces of real rank one, \emph{J. Funct. Anal.} 
\textbf{174} (2000), 274--300.

\bibitem[La]{La} R. Latter, A decomposition of $H^p(\BR^n)$ in terms
of atoms, \emph{Studia Math.} \textbf{62} (1978), 92--101.

\bibitem[Le]{Le} N.N.~Lebedev, \emph{Special functions and their applications},
Dover Publications, 1972.

\bibitem[Lo]{Lo} N. Lohou\'e, Transform\'ees de Riesz et fonctions sommables, \emph{Amer. J. Math.} \textbf{114} (1992),
875--922.

\bibitem[MaMV]{MMV} A. Martini, S. Meda and M. Vallarino, Hardy type spaces defined via maximal functions on certain nondoubling manifolds, preprint.

\bibitem[MOV]{MOV} A. Martini, A.~Ottazzi and M.~Vallarino, 
Spectral multipliers for sub-Laplacians on solvable
extensions of stratified groups,
\emph{J. Anal. Math.} \textbf{136} (2018), 357--397.

\bibitem[MMV1]{MMV0} G. Mauceri, S. Meda and M. Vallarino, 
Estimates for functions of the Laplace--Beltrami 
operator on manifolds with bounded geometry, \emph{Math. Res. Lett.} \textbf{16} (2009),  
861--879.

\bibitem[MMV2]{MMV1} G. Mauceri, S. Meda and M. Vallarino, 
Hardy-type spaces on certain noncompact manifolds and applications, 
\emph{J. London Math. Soc.} \textbf{84} (2011), 243--268.

\bibitem[MMV3]{MMV2} G. Mauceri, S. Meda and M. Vallarino, 
Atomic decomposition of Hardy type spaces on certain noncompact manifolds,
\emph{J. Geom. Anal.} \textbf{22} (2012), 864--891.

\bibitem[MMV4]{MMV4} G. Mauceri, S. Meda and M. Vallarino, 
Endpoint results for 
spherical multipliers
on noncompact symmetric spaces, 
\emph{New York J. Math.} \textbf{23} (2017), 1327--1356.

\bibitem[MVe]{MVe} S. Meda and G. Veronelli, Characterisation of Hardy-type spaces 
via Riesz transform on certain Riemannian manifolds, preprint.

\bibitem[MVo]{MVo} S. Meda and S. Volpi, Spaces of Goldberg type on 
certain measured metric spaces,  \emph{Ann. Mat. Pura Appl.} \textbf{196} (2017), 947--981.

\bibitem[SC]{SC} L. Saloff-Coste,
\emph{Aspects of Sobolev-type inequalities}, 
London Mathematical Society Lecture
Note Series \textbf{289}, Cambridge University Press, 2002.

\bibitem[SS]{SS} R.P. Sarkar and A. Sitaram,
The Helgason Fourier transform for symmetric spaces,
\emph{A tribute to C. S. Seshadri (Chennai, 2002)},
Trends Math., Birkh\"auser, 2003, 467--473.

\bibitem[T]{T} M.E. Taylor, Hardy spaces and bmo on manifolds with bounded geometry, 
\emph{J. Geom. Anal.} \textbf{19} (2009), 137--190.

\bibitem[U1]{Umax} A. Uchiyama,
A maximal function characterization of $H^p$ on the space of homogeneous type,
\emph{Trans. Amer. Math. Soc.} \textbf{262} (1980), 579--592.

\bibitem[U2]{U} A. Uchiyama,
A constructive proof of the Fefferman--Stein decomposition of $BMO(\BR^n)$, 
\emph{Acta Math.} \textbf{148} (1982), 215--241.

\bibitem[V]{V} M.~Vallarino, 
Spaces $H^1$ and $BMO$ on $ax+b$-groups, \emph{Collect. Math.} \textbf{60}
(2009), 277--295.

\bibitem[Vo]{Vo} S. Volpi, \emph{Bochner--Riesz means of eigenfunction expansions and 
local Hardy spaces on manifolds with bounded geometry}, Ph.D.\ Thesis, 
Universit\`a di Milano--Bicocca, 2012, \texttt{https://boa.unimib.it/handle/10281/29105}.

\bibitem[YZ1]{YZ1} D. Yang and Y. Zhou, Radial maximal function characterizations 
of Hardy spaces on RD-spaces and their applications, 
\emph{Math. Ann.} \textbf{346} (2010), no. 2, 307--333.

\bibitem[YZ2]{YZ2} D. Yang and Y. Zhou, Localized Hardy spaces $H^1$ related 
to admissible functions on RD-spaces and applications to Schr\"odinger operators, 
\emph{Trans. Amer. Math. Soc.} \textbf{363} (2011), 1197--1239.

\end{thebibliography}
\end{document}